\documentclass[12pt, oneside]{amsart}

\usepackage{geometry}                
\usepackage{graphicx}
\usepackage[mathscr]{euscript}
\usepackage{enumerate}
\usepackage{bbm}
\usepackage{amsthm,amssymb,amsmath}
\usepackage{subfigure}
\allowdisplaybreaks[4]
\usepackage{xcolor}

\begin{document}
\title[Model theory of complex numbers with polynomials functions]{Model theory of complex numbers with polynomials functions
}

\author{Benjamin Castle}
\address{Department of Mathematics, Ben Gurion University of the Negev, Be'er-Sheva 84105, Israel}
\email{bcastle@berkeley.edu}

\author{Chieu-Minh Tran}
\address{Department of Mathematics, National University of Singapore, Singapore}
\email{trancm@nus.edu.sg}

\thanks{The first author was supported NSF Grant DMS 1800692, ISF grant No. 555/21, and a Ben Gurion University Kreitman Fellowship.}
\thanks{The second author was supported by Anand Pillay's NSF Grant-2054271.}

\subjclass[2020]{Primary 03C45; Secondary 11B30}

\date{} 

\newtheorem{theorem}{Theorem}[section]
\newtheorem{lemma}[theorem]{Lemma}
\newtheorem{corollary}[theorem]{Corollary}
\newtheorem{fact}[theorem]{Fact}
\newtheorem{conjecture}[theorem]{Conjecture}
\newtheorem{proposition}[theorem]{Proposition}
\newtheorem{claim}[theorem]{Claim}
\theoremstyle{definition}
\newtheorem{definition}[theorem]{Definition}
\newtheorem{notation}[theorem]{Notation}
\newtheorem{convention}[theorem]{Convention}
\newtheorem{remark}[theorem]{Remark}
\newtheorem{assumption}[theorem]{Assumption}
\newtheorem*{thm:associativity}{Theorem \ref{thm:associativity}}
\newtheorem*{thm:associativity2}{Theorem \ref{thm:associativity2}}
\newtheorem*{thm:associativity3}{Theorem \ref{thm:associativity3}}
\def\tri{\,\triangle\,}
\def\P{\mathbb{P}}

\def\E{\mathbb{E}}

\def\e{\mathbbm{1}}
\def\h{\mathrm{hdim}}
\def\ndim{\mathrm{ndim}}
\def\d{\,\mathrm{d}}
\def\dd{\mathfrak{d}}
\def\supp{\mathrm{supp}}
\def\BM{\mathrm{BM}}
\def\RR{\mathbb{R}}
\def\CC{\mathbb{C}}
\def\TT{\mathbb{T}}
\def\ZZ{\mathbb{Z}}
\def\X{\widetilde{A}}
\def\Y{\widetilde{Y}}
\def\XX{\alpha}
\def\YY{\beta}
\def\L{L^+}
\def\x{\widetilde{x}}
\def\SL{\mathrm{SL}}
\def\PSL{\mathrm{PSL}}
\def\mm{\mathrm{mod}}
\def\Stab{\mathrm{Stab}}
\def\acl{\operatorname{acl}}

\newcommand\NN{\mathbb N}
\newcommand{\Case}[2]{\noindent {\bf Case #1:} \emph{#2}}
\newcommand\inner[2]{\langle #1, #2 \rangle}

\def\ben#1{{\color{red}{({\sc Ben says: }{\sf #1})}}}
\def\minh#1{{\color{blue}{({\sc Minh says: }{\sf #1})}}}

\newcommand{\GL}{\mathrm{GL}}
\newcommand{\OO}{\mathrm{O}}
\newcommand{\Ad}{\mathrm{Ad}}

\newcommand{\Ker}{\mathrm{Ker}}
\newcommand{\Lg}{\mathfrak{g}}
\newcommand{\Lh}{\mathfrak{h}}
\newcommand{\Lk}{\mathfrak{k}}
\newcommand{\La}{\mathfrak{a}}
\newcommand{\Ln}{\mathfrak{n}}
\newcommand{\Lp}{\mathfrak{p}}

\newcommand{\sM}{\mathscr{M}}

\begin{abstract}

Let $\CC$ be the set of complex numbers, and let $\mathcal P$ be a collection of complex polynomial maps in several variables. Assuming at least one $P\in\mathcal P$ depends on at least two variables, we classify all possibilities for the structure $(\mathscr M;\mathcal P)$ up to definable equivalence. In particular, outside a short list of exceptions, we show that $(\mathscr M;\mathcal P)$ always defines $+$ and $\times$. Our tools include Zilber's Restricted Trichotomy, as well as the classification of symmetric non-expanding pairs of polynomials over $\mathbb C$ from arithmetic combinatorics. Along the way, we also give a new condition for a reduct $\mathscr M=(M,...)$ of a smooth curve over an algebraically closed field to recover all constructible subsets of powers of $M$.

\end{abstract}
\maketitle

\section{Introduction}


\subsection{Background and result} Suppose $P$ and $Q$ are complex polynomial maps in several variables. Using mathematical logic, one can formalize the notion `$Q$ can be defined in terms of $P$': this says that the graph of $Q$ is first-order definable in the structure on $\mathbb C$ with the graph of $P$ as its only basic relation. For example, one can show as an exercise that $x+2y$ and $x-y$ can be defined in terms of $x+y$, neither of $x+y$ and $x\times y$ can be defined in terms of the other, and both $x+y$ and $x\times y$ can be defined in terms of $x^2+y^2$. Notice that since $+$ and $\times$ generate all polynomials, this implies that \textit{every} polynomial can be defined in terms of $x^2+y^2$ alone. 

One might wonder, for arbitrary $P$ and $Q$, under what conditions can $Q$ be defined in terms of $P$? In this paper, we give a surprising answer: \textit{almost always}. In particular, outside a short list of exceptions, every polynomial is by itself capable of defining $+$ and $\times$ (Theorem \ref{single poly main thm}). The list of exceptions roughly contains unary polynomials, linear polynomials, and `monomials up to a translation.' We will see that linear polynomials are interdefinable with vector spaces, and monomials are interdefinable with multiplication. Thus, in a sense, the negative example given above (addition and multiplication cannot define each other) is canonical: for polynomials in two or more variables, \textit{every counterexample} essentially arises from either addition or multiplication. Despite the concrete nature of the result, our proof is almost entirely abstract, and relies on deep results in model theory and combinatorics.

 Let us give more details. A \textbf{reduct} of a structure $\mathscr M=(M; \ldots)$ is, in this paper, another structure $\mathscr M'$ with the same underlying set $M$, such that every set definable (with parameters) in $\mathscr M'$  is also definable in $\mathscr M$.
If $\mathscr M'$ is a reduct of $\mathscr M$, we also say $\mathscr M$ is an \textbf{expansion} of $\mathscr M'$.
Two structures are \textbf{definably equivalent} if each one is a reduct of the other. We are more generally interested in the following problems:
\begin{enumerate}
    \item Determine when a reduct of the field $(\CC; +, \times)$ of complex numbers defines $+$ and $\times$ (in other words, the reduct is definably equivalent to $(\CC; +, \times)$ itself).
    \item Classify the reducts of $(\CC; +, \times)$ up to definable equivalence.
\end{enumerate}

 In this paper, we consider the \textbf{polynomial reducts} of $\mathbb C$ -- structures of the form $(\mathbb C;\mathcal P)$ where $\mathcal P$ is a collection of polynomial maps $P:\mathbb C^{n_P}\rightarrow\mathbb C$. Notice that this setting encompasses the definability question for single polynomials described at the beginning of the introduction. In particular, the earlier `$P$ and $Q$ can be defined in terms of each other', more precisely, means $(\CC; P)$ is definably equivalent to $(\CC; Q)$.
 
 Our main theorem below answers Question (1) completely for polynomial reducts, and answers Question (2) for polynomial reducts not definable exclusively by unary polynomials. For instance, as hinted at above, our theorem implies immediately that the reduct $(\CC; (x,y) \mapsto x^{17}+x^6y^{8}-y^3)$ defines both $+$ and $\times$. 

 \begin{theorem} \label{thm: mainclassification}
Suppose $\mathcal P$ is a collection of complex polynomial maps, and let $\mathscr M=(\mathbb C;\mathcal P)$. Then $\mathscr M$ is interdefinable with exactly one of following:
\begin{enumerate}[{\rm (i)}]
    \item $(\mathbb C;\mathcal U)$, where $\mathcal U$ is a collection of unary polynomials.
    \item $(\CC; +, (\lambda\cdot)_{\lambda \in F}) $ where $F$ is a subfield of $\CC$ (that is, the $F$-vector space structure on $\mathbb C$).
    \item $(\CC;\times_r)$ for some (unique) $r\in\mathbb C$, where $\times_r$ is the operation $x\times_ry=(x-r)(y-r)+r$.
    \item The full reduct $(\CC;+,\times)$.
\end{enumerate}
Moreover, (i) happens if and only if each $P\in\mathcal P$ involves at most one variable. (ii) happens if and only if (i) fails and each $P\in\mathcal P$ is linear, and in this case $F$ is generated by all coefficients on all variables appearing in the $P\in\mathcal P$. (iii) happens if and only if (i) fails and each $P\in\mathcal P$ is a monomial in the group operation $\times_r$. And (iv) happens in every other case.
\end{theorem}
 
Questions (1) and (2) above were first studied by Martin~\cite{Martinwho}, who classified expansions of $(\CC;+)$ and $(\CC;\times)$ by a single rational function. Subsequent results on expansions of $(\CC;+)$ were obtained by Zilber and Rabinovich~\cite{RZ}, and Marker and Pillay~\cite{MarkerPillay}. In particular, Marker and Pillay proved that every reduct of $(\CC;+,\times)$ expanding $(\CC;+)$ is definably equivalent to either the full field $(\CC;+,\times)$, or to the vector space structure $(\CC; +, (\lambda \cdot)_{\lambda \in F})$ over a subfield $F $ of $\CC$. Recently, Hasson and the first author~\cite{AssafBen} gave a complete answer to (1) in the form of an abstract condition called \textit{very ampleness}. As one concrete application, they gave a unified proof of the Marker--Pillay result and its multiplicative counterpart, showing that a reduct of $(\CC; +, \times)$ containing $\times$ must be definably equivalent to either $(\CC; \times)$ or $(\CC; +, \times)$.

Meanwhile, if one does not assume $+$ or $\times$ is contained in the reduct, concrete classifications have only been attained `up to finite covers' (see the next two paragraphs, and also \cite{Loveys}). As reducts generated by polynomials functions are quite general,
Theorem~\ref{thm: mainclassification} can be seen as a step toward classifying all reducts up to definable equivalence. In fact, we see it as evidence that the reducts of $(\CC; +, \times)$ can be classified `up to a unary function and trivial noise'-- namely there is a short list of structures on $\mathbb C$ such that for any reduct $\mathscr M$ of $(\CC; +, \times)$, there is a function $f:\mathbb C\rightarrow\mathbb C$ such that the image of $\mathscr M$ under $f$ is in the given list, and $\mathscr M$ is determined up to definably equivalence by $f$, its image under $f$, and some extra noise coming from binary relations (the analog of our unary polynomials).

Question (1) is closely related to Zilber's Restricted Trichotomy Conjecture for $(\CC; +, \times)$ (see Section~2 for the precise statement); indeed, this conjecture will play a crucial role in our proof. Restricted to our setting, Zilber's conjecture predicts that \textbf{non-local modularity} (an abstract model-theoretic condition) is necessary and sufficient for a reduct of $(\CC; +, \times)$ to define an isomorphic copy of $(\CC;+,\times)$. For $(\CC;+,\times)$, this was recently verified in full generality by the first author (\cite{Ben}). In our setting (reducts of the field itself), the conjecture is also implied by previous partial cases due to Rabinovich \cite{Seriousmistake} and Hasson-Sustretov~\cite{AssafDmitry}.

In general, if a reduct $\mathscr M$ of $(\CC;+\times)$ defines an isomorphic copy of $(\CC;+,\times)$ as above, it must define $+$ and $\times$ `up to a unary function' -- that is, there is a non-constant rational function $f:\CC\rightarrow\CC$ such that $\mathscr M$ defines the relations $f(x)+f(y)=f(z)$ and $f(x)f(y)=f(x)$. However, in some cases the function $f$ cannot be eliminated. From this angle, Question (1) is the next step in the same direction -- we want to classify structures that interpret not just \textit{any} copy of $(\CC,+,\times)$, but \textit{the} copy (i.e. the original field).

\subsection{Outline of proofs and organization of the paper}

The proof of Theorem \ref{thm: mainclassification} is organized according to the Zilber trichotomy for strongly minimal structures. 

An infinite definable set is \textbf{strongly minimal} if all of its definable subsets, even in elementary extensions, are finite or cofinite. A structure is strongly minimal if its underlying set is a strongly minimal set. The complex field is strongly minimal, and thus so are all of its reducts.

Zilber proposed a `coarse' classification of strongly minimal structures `up to finite correspondence.' This is a generalization of `up to definable bijection': two definable sets $X$ and $Y$ are in finite correspondence if there is a definable set $C\subseteq X\times Y$ such that both projections $C\rightarrow X$ and $C\rightarrow Y$ are finite-to-one with cofinite image.

Zilber considered three possibilities for a strongly minimal structure $\mathscr M$. We will not give precise definitions yet.\\
\textbf{Trivial case:} Every definable relation $X\subseteq M^n$ is, roughly, determined by relations in $M$ and $M^2$.\\
\textbf{Group case:} Up to a finite correspondence, $\mathscr M$ is an abelian group equipped only with certain subgroups of powers of $M$.\\
\textbf{Field case:} Up to a finite correspondence, $\mathscr M$ is an algebraically closed field equipped only with the field operations.

\begin{remark} The union of the trivial and group cases is known in model theory as \textbf{local modularity}. The group case is often called `non-trivial locally modular.' The structures $(\CC;+)$ and $(\CC;\times)$ are non-trivial locally modular. The structure $(\CC;+,\times)$ is not locally modular.
\end{remark}

The content of Zilber's Restricted Trichotomy Conjecture, in our setting, is that every reduct of $(\mathbb C;+,\times)$ belongs to one of these three cases. Our theorem then says that, for polynomial reducts, we can get rid of the finite correspondences. The result is a more precise classification of each case. In particular, we will show that the trivial case is equivalent to (i) in Theorem \ref{thm: mainclassification}; the group case is equivalent to the union of (ii) and (iii); and the field case is equivalent to (iv). We prove each of these statements by a separate argument.



The trivial case can be analyzed quite easily. Any polynomial depending on at least two variables immediately violates triviality. Meanwhile, it is a short exercise to show the converse -- that if $\mathscr M$ is a reduct given by unary polynomials, $\mathscr M$ is trivial. This analysis is done in Section 3. Note that we do not obtain a precise classification of the interdefinability relation between trivial polynomial reducts. However, a full classification is obtained in Corollary \ref{C: trivial interdefinability} for reducts defined by a single unary polynomial.

The group case is treated in Sections 4, 5, and 6. Section~4 verifies that, assuming non-triviality, linear polynomials are interdefinable with vector spaces, and monomials are interdefinable with multiplication. This is again a straightforward exercise. Conversely, Sections 5 and 6 show that polynomial reducts in the group case always arise from linear polynomials and `twisted monomials' (monomials in the group $\times_r$ from Theorem \ref{thm: mainclassification}). The highlight of these sections is the use of arithmetic combinatorics to `eliminate' the finite correspondence with an abelian group, thereby gaining precise information about which polynomials can appear in the reduct.

Let us elaborate. We are given a polynomial reduct $\mathscr M=(\mathbb C;\mathcal P)$, an abelian group $G$, and a definable finite correspondence $C\subseteq\CC\times G$. We want to show that all $P\in\mathcal P$ are linear or monomials in some $\times_r$. We know that all definable sets in all $G^n$ are `group-like' (formally, they are Boolean combinations of cosets of subgroups). Thus, we would like to transfer group-like data `through $C$', to recover a group operation on $\CC$. We do this using asymptotic finite combinatorics. The idea is that when passing a finite set through the correspondence $C$ (say from $G$ to $\CC$), its cardinality will only be scaled by a constant factor; so $C$ does not change any asymptotic growth rates, and thus any `group-like' combinatorial phenomena in $G$ should also occur in $\mathbb C$. 

It will be easy to reduce to the case that $\mathcal P$ consists of a single binary polynomial, say $P(x,y)$. In this case, the image of the graph of $P$, through $C$, is some definable set $Y\subseteq G^3$, which must be `group-like'. For simplicity, let us assume here that $Y$ is the graph of a homomorphism $\phi:G^2\rightarrow G$ (the general case is only slightly more complicated). Then, using known techniques, for all $\epsilon$ one can find arbitrary large finite $B \subseteq G$ (essentially a geometric progression) such that $$|\phi(B\times B)| < |B|^{1+\epsilon}.$$ In our terminology, we say that $\phi$ has \textbf{symmetric small expansion}.

Setting $A$ to be the preimage of $B$ in $\CC$ (and with a little more work), one then obtains $$|P(A\times A)| < |A|^{1+\epsilon}$$ -- thus $P$ also has symmetric small expansion. Then, by Elekes--Ronyai-type theorems from arithmetic combinatorics, we get strong information about $P$: there are unary polynomials $f$ and $u$ such that either $P(x,y)=f(c_1u(x)+c_2u(y))$ for some $c_1,c_2\in\CC$ or $P(x,y)=f(u^m(x) u^n(y))$ for some integers $m,n$. With a bit more work, we deduce all possibilities for $f$ and $u$, concluding that $P$ has one of the desired forms. 


Section 7 treats the most interesting field case. For this, we actually prove a more general statement (Theorem \ref{nlm case}): Suppose $\sM=(M;...)$ is a strongly minimal reduct of a curve over an algebraically closed field, which is not locally modular (i.e. belongs to the `field case'). Then $\sM$ recovers all constructible subset of powers of $M$, provided it defines (roughly) both a binary function $M^2\rightarrow M$ and a generically $d$-to-one function $M\rightarrow M$ for some $d\geq 2$. It is straightforward to build such functions in a polynomial reduct (Theorem \ref{nlm classification}).

Let us describe the proof of Theorem \ref{nlm case}. We are given a reduct $\sM=(M;...)$ of a curve, and a finite correspondence with a definable algebraically closed field $(F,\oplus,\otimes)$. By \textit{elimination of imaginaries} in $F$, there is even a finite-to-one definable function $f:M\rightarrow F^n$ for some $n$. We think of $f$ as realizing $M$ as a finite cover of a curve over $F$. It is well known (and proven explicitly in \cite{AssafBen}) that if $f$ is injective (i.e. the cover is trivial), $\mathscr M$ defines $+$ and $\times$. So our goal is to show that some such $f$ is injective.

We use a generalization of an argument due to Hrushovski, for analyzing (and reducing) finite covers of definable sets. Past iterations of Hrushovski's technique assume the `cover' (our $M$) is a definable group. The idea is to choose a cover of minimal index, and show by minimality that the fibers are cosets of a finite subgroup. If the group is divisible, one can then construct a trivial cover, thereby concluding the minimal cover was always trivial.

In our case, we have no definable group operation on $M$, so we need a more elaborate analysis of the fibers in the minimal cover $f$. We define the equivalence relation $\sim$ on $M$ by $x\sim y$ if $f(x)=f(y)$. We then use minimality to show that $\sim$ is respected (in a precise sense) by all definable finite correspondences between $M$ and itself (Lemma \ref{plane curves always respect equivalence relation}). Finally, we conclude that $f$ is injective via a counting argument, applying the above to various correspondences built from the provided binary and unary functions.

\subsection{Notation and convention} We work throughout with strongly minimal structures $\mathscr M=(M;...)$ in a language $\mathcal L$. If there is ambiguity in the choice of structure, we may add to all terms defined below the prefix $\mathscr M$ (creating $\mathscr M$-definable, $\mathscr M$-generic, etc.). \textbf{However, unless otherwise stated, the reader should assume all model-theoretic terms refer to the reducts of algebraic curves we consider (not to the full field structure).}

By a \textbf{definable} set, we refer to definability over parameters in the expansion $\mathscr M^{\textrm{eq}}$, obtained by closing the class of definable sets under quotients by definable equivalence relations. 

We denote sets of parameters from $M$ by $A$, $B$, .... Unless otherwise stated, sets of parameters are assumed to \textbf{small}, i.e. of smaller cardinality than $|M|$.

\section{Preliminaries}

\subsection{Some model theory}
Recall that $\mathscr M$ is \textbf{saturated} if whenever $\phi(\bar x,A)$ is a consistent set of formulas in the variable $\bar x$ with parameters from the small set $A$, then $M$ contains a realization of $\phi$. Every strongly minimal structure has a saturated elementary extension; every uncountable strongly minimal structure in a countable language is already saturated. We will almost always be able to assume $\mathscr M$ is saturated, but will make this explicit as we go.

Recall that Morley rank gives a well-behaved dimension theory for definable sets in $\mathscr M$ (for reducts of $(\CC;+,\times)$, Morley rank agrees with the usual dimension theory of varieties). For example, $\dim(M^n)=n$ for each $n$, and the 0-dimensional sets are precisely the non-empty finite sets. 

A definable set of dimension $n$ is called \textbf{stationary} if it is not the disjoint union of two definable sets of dimension $n$. Every definable set $X$ is a finite disjoint union of stationary sets (of the same dimension as $X$). These sets are unique up to almost equality (defined below), and are called the \textbf{stationary components} of $X$. 

For a tuple $a$ and set $A$, the notation $\dim(a/A)$ denotes the smallest dimension of a set containing $a$ and definable with parameters from $A$. One then says that $a_1,...,a_n$ are \textbf{independent} over $A$ if $\dim(a_1,...,a_n/A=\dim(a_1/A)+...+\dim(a_n/A)$. If $\dim(a/A)=0$ we say that $a$ is \textbf{algebraic} over $A$, denoted $a\in\acl(A)$.

 If $X\subseteq Y$ are definable, we say that $X$ is \textbf{generic} in $Y$ if $\dim(X)=\dim(Y)$, \textbf{small} $Y$ if $\dim(X)<\dim(Y)$, and \textbf{large} in $Y$ if $\dim(Y-X)<\dim Y$. One can similarly notions of \textbf{almost containment} and \textbf{almost equality} of definable sets. We then say a definable function $f:X\rightarrow Y$ is \textbf{almost surjective} if $\operatorname{Im}(f)$ is large in $Y$, and \textbf{almost finite-to-one} if the union of all infinite fibers of $f$ is small in $X$. 
 
 If $X$ is definable with parameters from $A$, and $a\in X$, we say that $a$ is \textbf{generic in} $X$ \textbf{over} $A$ if $\dim(a/A)=\dim(X)$. Generic points always exist if $\mathscr M$ is saturated and $A$ is small.
 
 \subsection{The Zilber Trichotomy}

 We now explicitly state the Zilber trichotomy. 

\begin{definition} Suppose $\mathscr M=(M,...)$ is strongly minimal.
\begin{itemize}
    \item If $\mathscr M$ is saturated, we say $\mathscr M$ is \textbf{trivial} if whenever $A\subseteq M$, $b\in M$, and $a\in\acl(A)$, there is some $a\in A$ with $b\in\acl(a)$.
    \item If $\mathscr M$ is not saturated, we say $\mathscr M$ is trivial if some (equivalently any) saturated elementary extension of it is trivial.
\end{itemize}
\end{definition}

\begin{definition} Suppose $\mathscr M=(M,...)$ is strongly minimal. Then $\mathscr M$ is \textbf{not locally modular} if there is a definable 2-dimensional family of plane curves in $\mathscr M$. That is, a definable set $T$ of dimension 2, and a definable set $C\subseteq M^2\times T$, such that:
\begin{itemize}
    \item For each $t\in T$, $\dim(C_t)=1$.
    \item For any $t\neq t'\in T$, $\dim(C_t\cap C_t')=0$.
\end{itemize}
Otherwise, we say $\mathscr M$ is \textbf{locally modular}.
\end{definition}
    
Strongly minimal structures are typically organized into three levels: trivial, non-trivial locally modular, and non-locally modular. The idea is that these correspond to `graph-like,' `group-like,' and `field-like.' For the latter two levels, the terminology is no coincidence:

\begin{fact}\cite{Hrulocmod},\cite{HrPi87}\label{nontrivial locally modular characterization} Suppose $\mathscr M=(M;...)$ is strongly minimal, non-trivial, and locally modular. Then there is a definable, strongly minimal, abelian group $(G;+,...)$ in definable finite correspondence with $\mathscr M$. Moreover, for any such $G$, every definable subset of every $G^n$ is a finite Boolean combination of cosets of definable subgroups of $G^n$.
\end{fact}

\begin{fact}[Zilber's Restricted Trichotomy for $\CC$ \cite{Ben}]\label{Ben} Suppose $\mathscr M=(M;,...)$ is strongly minimal and not locally modular. Suppose moreover that the underlying set $M$, and all definable subsets of all $M^n$, are complex constructible sets (Boolean combinations of affine varieties). Then there is a definable field $(F;\oplus,\otimes)$, isomorphic to $(\CC;+,\times)$, in definable finite correspondence with $M$. 
\end{fact}


\section{Trivial Reducts}

In this section we analyze trivial polynomial reducts. This material is quite straightforward, so we only sketch the proofs.

\begin{lemma}\label{trivial implies unary} Let $\mathscr M=(\mathbb C;...)$ be a trivial reduct of $(\mathbb C,+,\times)$, and let $P\in\mathbb C[x_1,...,x_n]$ be an $\mathscr M$-definable polynomial. Then $P\in\mathbb C[x_i]$ for some $i$.
\end{lemma}
\begin{proof} Suppose $P$ is definable over $A$. Replacing $\mathscr M$ with $(\CC;P)$, we may assume $\mathscr M$ is saturated, so there is an element $(a_1,...,a_n)\in\CC^n$ which is generic over $A$. Now assume $P\notin\mathbb C[x_i]$ for any $i$. Then $P(a_1,...,a_n)\in\operatorname{acl}(A,a_1,...,a_n)$ but $P(a_1,...,a_n)\notin\operatorname{acl}(A,a_i)$ for any $i$, which easily contradicts triviality.
\end{proof}

By Lemma \ref{trivial implies unary}, we may restrict our attention to unary polynomials, i.e. those of the form $P(x)\in\mathbb C[x]$. For such polynomials, we now show the converse:

\begin{lemma}\label{unary implies trivial} Let $\mathcal P$ be a collection of unary complex polynomials. Then $\mathscr M=(\mathbb C;\mathcal P)$ is trivial. In fact, in $\mathscr M$, algebraic closure is just the union of the constants in $\mathcal P$ and the closure under $P$ and $P^{-1}$ for non-constant $P\in\mathcal P$.
\end{lemma}
\begin{proof} We may assume $\mathcal P$ is finite, so $\mathscr M$ is saturated. Draw an edge-colored directed graph $G$ with vertex set $\mathbb C$, one color for each $P\in\mathcal P$, and an edge of color $P$ between $x$ and $y$ whenever $P(x)=y$. 
So any color-preserving automorphism of $G$ is an automorphism of $\mathscr M$. Moreover, for $x\in\mathbb C$, the isomorphism type of the connected component of $G$ containing $x$ is encoded into the type of $x$ over $\emptyset$. It follows easily that for any three generics $x,y,z\in\mathbb C$ in different components, there is an automorphism of $\mathscr M$ fixing $x$ and sending $y$ to $z$; in particular, this gives $y\notin\acl(x)$. It follows easily that algebraic closure in $\mathscr M$ is as described in the lemma, and triviality is immediate from this description of algebraic closure.
\end{proof}

So triviality is characterized by unary polynomials. This is all we really need here for our main result; however, for the case of a single polynomial, we can give the following more precise classification.

In what follows, for $P(x)\in\mathbb C[x]$ the notation $P^n$ will denote the $n$-fold composition of $P$. In particular, if $n=0$ this means the identity map, and if $n$ is negative this means the $-n$-fold composition of the inverse of $P$ (when it exists). 

\begin{theorem}\label{T: one polynomial trivial case} Let $P\in\mathbb C[x]$, and let $\mathscr M=(\mathbb C;P)$. Then the definable unary rational functions in $\mathscr M$ are precisely the following:
\begin{enumerate}
    \item For constant $P$, the constant and identity maps.
    \item For $\deg(P)=1$, the constant maps and those of the form $P^n$ for $n\in\mathbb Z$.
    \item For $\deg(P)=2$, the constant maps and those of the form $P^n$ or $r\circ P^n$ for $n\geq 0$, where $r$ is the reflection across the axis of $P$; that is, if $P(x)=ax^2+bx+c$, then $r(x)=\frac{-b}{a}-x$.
    \item For $\deg(P)\geq 3$, the constant maps and those of the form $P^n$ for $n\geq 0$.
\end{enumerate}
\end{theorem}
\begin{proof} This is an exercise in quantifier elimination. If $P$ is constant, the identity, or an involution $x\mapsto c-x$, then everything is clear. So assume $P$ has degree $d\geq 1$ and is not the identity or an involution. It follows easily that any two iterates of $P$ are distinct.

Let $A\subseteq\mathbb C$ be a countable algebraically closed subfield containing the coefficients of $P$. Consider the language $\mathcal L$ containing constant symbols $c_a$ for each element $a\in A$, and a unary function symbol $f$. Let $T$ be the theory asserting each of the following:
\begin{enumerate}
    \item The atomic diagram of $A$ in the structure $\mathscr M$.
    \item $c_{a_i}$ has exactly $l_i$ preimages under $f$ for each $i$, where $a_1,...,a_n$ are the elements of $\mathbb C$ which do not have exactly $d$ preimages under $P$ (note that there are finitely many such elements, each of which belongs to $A$), and where $a_i$ has exactly $l_i$ preimages under $P$.
    \item Every element other than $c_{a_1},...,c_{a_n}$ has exactly $d$ preimages under $f$.
    \item If $f^i(x)=f^j(x)$ for some $i\neq j\geq 0$, then $x=c_a$ for one of the finitely many elements $a$ such that $P^i(a)=P^j(a)=0$ (here was use that since any two iterates of $P$ are distinct, for each $i$ and $j$ there are only finitely many such elements, each of which belongs to $A$).
\end{enumerate}
It is easy to see that $T$ is complete and eliminates quantifiers. Clearly $\mathscr M\models T$, so $T$ completely axiomatizes $\operatorname{Th}(\mathscr M)$. In particular, $T$ is strongly minimal.

In general, a model of $T$ is given by a copy of $A$ together with a disjoint set on which $f$ induces a `generic' $d$-to-1 function (i.e. an everywhere $d$-to-1 map for which $f^i(x)=f^j(x)$ has no solutions unless $i=j$). Let $a\in\mathbb C$ be generic. To determine the definable rational functions in $\mathscr M$, it suffices to determine $\operatorname{dcl}(a)$ (i.e. those elements of $\mathbb C$ definable in $\mathscr M$ over the parameter $a$). Now as in the previous lemma, we draw a graph on $\mathbb C$ with an edge between $x$ and $y$ whenever $f(x)=y$. Then, similarly, any two connected components not containing any $c_a$'s are isomorphic. So $\operatorname{dcl}(a)$ is restricted to the connected component of $a$, on which we just have a generic $d$-to-1 function. By further analyzing automorphisms of such functions, one can easily reduce $\operatorname{dcl}(a)$ to the four cases in the theorem. For clarity, we note that if $d=2$ the map $r$ is respresented by the unique $a'\neq a$ with $f(a)=f(a')$.
\end{proof}

Theorem \ref{T: one polynomial trivial case} implies a classification of the interdefinability relation on unary polynomials:

\begin{corollary}\label{C: trivial interdefinability} Let $P,Q\in\mathbb C[x]$. Then $(\mathbb C;P)$ is interdefinable with $(\mathbb C;Q)$ if and only if one of the following holds:
\begin{enumerate}
    \item $P$ and $Q$ are both among the identity and constant maps.
    \item $P$ and $Q$ are inverse linear maps.
\end{enumerate}
\end{corollary}

\section{Non-trivial Locally Modular Reducts}

Over the next three sections, we analyze non-trivial locally modular reducts. Due to Theorem \ref{unary implies trivial}, we will always assume at least one of the polynomials in the signature depends on at least two variables. The current section introduces those non-trivial locally modular which ultimately turn out to be the only ones.

To state our results, it is convenient to introduce the following:

\begin{definition} Given $r\in\mathbb C$, we define the operation of \textbf{multiplication twisted by $r$} on $\mathbb C$ by $x\times_ry=(x-r)(y-r)+r$.
\end{definition}

Note that the equation $x\times_ry=z$ is equivalent to $(x-r)(y-r)=z-r$. Thus it is clear that $\times_r$ defines a group operation on $\mathbb C-\{r\}$, which is isomorphic to the multiplicative group of $\mathbb C$.

\begin{definition} Let $r\in\mathbb C$. Then by a \textbf{monomial twisted by $r$}, we mean a polynomial of the form $a\times_r{x_1}...\times_r{x_n}$, where $a\in\mathbb C$ and the $x_i$'s are variables which are allowed to repeat. In general we say that a polynomial is a \textbf{twisted monomial} if it is a monomial twisted by $r$ for some $r\in\mathbb C$.
\end{definition}

The main result of this section is that linear polynomials and twisted monomials are interdefinable, respectively, with vector spaces and twisted multiplication:

\begin{proposition}\label{P: loc mod examples} Let $\mathcal P$ be a collection of complex polynomials, at least one of which depends on at least two variables. Let $\mathscr M=(\mathbb C;\mathcal P)$.
\begin{enumerate}
    \item Suppose each $P\in\mathcal P$ is linear, so has the form $$P(x_1,...,x_{n_P})=b_P+\sum_{i=1}^{n_P}a_i^Px_i.$$ for some constants $b_P$ and $a_i^P$. Then $\mathscr M$ is interdefinable with the vector space structure on $\mathbb C$ over the field $F=\mathbb Q(\{a_i^P\}_{P\in\mathcal P,i\leq n_P})$.
    \item Suppose there is some $r\in\mathbb C$ such that each $P\in\mathcal P$ is a monomial twisted by $r$. Then $\mathscr M$ is interdefinable with $(\mathbb C,\times_r)$.
\end{enumerate}
In particular, if either (1) or (2) holds, then $\mathscr M$ is locally modular.
\end{proposition}
\begin{proof} In each case, it is clear that $\mathscr M$ is definable from the desired structure. We show the converse.
\begin{enumerate}
    \item Let us fix some $P\in\mathcal P$ which depends on at least two variables, and write it as $P(x_1,...,x_n)=a_1x_1+...+a_nx_n+b$. Without loss of generality we assume $a_1,a_2\neq 0$. Now since $a_2\neq 0$, we can find a tuple $x_2,...,x_n$ with $a_2x_2+...+a_nx_n+b=0$. Specializing to this tuple, we obtain the map $x\mapsto a_1x$, and therefore also $x\mapsto\frac{x}{a_1}$ by inverting. By a similar argument, we obtain $x\mapsto a_j$ for each $j$, and therefore also $x\mapsto\frac{x}{a_j}$ for all non-zero $a_j$. Now precomposing these inverted scalings with $P$, we get the map which replaces each non-zero $a_j$ with 1. Assume this map has the form $x_1+...+x_k+b$, where $2\leq k\leq n$. Specializing all but one $x_j$ to 0, we obtain $x\mapsto x+b$, and therefore by inverting we get $x\mapsto x-b$. Now replacing $x_k$ with $x_k-b$, we obtain $x_1+...+x_k$; and thus specializing all but two variables to 0, we get addition. Then we can now subtract $b_i$ from each $P_i$, and assume each $P_i$ has no constant term. Then specializing all but one variable to 0 in each $P_i$, we get the scaling by each $a_j^i$. Combined with addition, this is enough to recover the $F$-vector space structure.
    \item Without loss of generality we assume $r=0$, so each $P\in\mathcal P$ is a monomial. Let us again distinguish some $P\in\mathcal P$ which depends on at least two variables, and write it as $P(x_1,...,x_n)=bx_1^{a_1}...x_n^{a_n}$, where $b\neq 0$ each $a_j$ is a non-negative integer. Without loss of generality we assume $a_1,a_2\geq 1$. Then by specializing all other variables to 1, we obtain the map $(x_1,x_2)\mapsto bx_1^{a_1}x_2^{a_2}$. Without loss of generality let us identify this map as $P$; our task is now to show that multiplication is definable from $P$ alone. Now by making appropriate specializations as in (1), we recover the power maps $x^{a_1}$ and $x^{a_2}$. Then the map $(x_1,x_2)\mapsto bx_1x_2$ is definable, because $y=bx_1x_2$ holds if and only if there are $x_1',x_2'$ with $(x_1')^{a_1}=x_1$, $(x_2')^{a_2}=x_2$, and $y=b(x_1')^{a_1}(x_2')^{a_2}$. Finally, setting $x_2=1$ in $y=bx_1x_2$ gets scaling by $b$, and therefore scaling by $\frac{1}{b}$. Then replacing $x_2$ with $\frac{x_2}{b}$ in $y=bx_1x_2$ gets multiplication.
\end{enumerate}
\end{proof}

For completeness, we also point out:

\begin{lemma} Let $\mathcal L$ be the family of languages on $\mathbb C$ containing each of the following:
\begin{enumerate} 
\item For each subfield $F\subseteq\mathbb C$, the $F$-vector space language $L_F$ (i.e. addition, and scaling maps for each element of $F$).
\item For each $r\in\mathbb C$, the language $L(r)$ of $r$-twisted multiplication (i.e. with the sole binary operation $\times_r$).
\end{enumerate}
Then no two languages in $\mathcal L$ are interdefinable. 
\end{lemma}
\begin{proof}
It is well known that the $L$-structure on $\mathbb C$ eliminates quantifiers for each $L\in\mathcal L$, making the lemma obvious. 
\end{proof}

\section{Small expansion in strongly minimal sets}

Our next goal is to show that every non-trivial locally modular polynomial reduct is interdefinable with either a vector space or a twisted multiplication. In this section, we develop analogs for strongly minimal sets of some tools from additive combinatorics. In the next section, we apply these tools to `recognize' addition and multiplication abstractly from the asymptotic combinatorics provided by local modularity. The material in this section is technical but straightforward. The idea is to generalize the machinery of \textit{small expansion} from functions to multivalued functions.

Recall that a polynomial $P(x,y)$ over $\mathbb R$ or $\mathbb C$ is said to have \textbf{small expansion} if for all $\epsilon>0$ one can find arbitrarily large positive integers $N$ and finite sets $A$ and $B$ with $|A|=|B|=N$ and $|P(A,B)|<N^{1+\epsilon}$. One also has a similar notion of \textbf{symmetric small expansion}, obtained by further imposing that $A=B$. Polynomials with symmetric small expansion were classified in \cite{JTR}, essentially as those `coming from' either addition or multiplication in a precise sense. Our goal is to show that an analog of symmetric small expansion holds in certain strongly minimal groups, and is moreover preserved through finite correspondences with these groups.

\begin{assumption} Throughout Section 5, we work in a sufficiently saturated model $\mathscr M$ of a complete theory $T$. In particular, a strongly minimal set is a set $D$ definable in $\mathscr M$ such that $D$ is strongly minimal. One can further assume $\mathscr M$ is strongly minimal too, but does not need to.
\end{assumption}

\begin{definition} Let $D$ and $E$ be strongly minimal sets, and let $X\subseteq D^2\times E$ be definable. We say that $X$ is a \textbf{quasi-function} from $D^2$ to $E$ if $\dim(X)=2$, $X$ is stationary, and each of the projections $X\rightarrow D^2$ and $X\rightarrow D\times E$ (for both copies of $D$) is almost finite-to-one.
\end{definition}

Quasi-functions should be thought of as abstractions of binary polynomials on $D$. We now introduce symmetric small expansion for quasi-functions.

\begin{notation} Let $D$ and $E$ be strongly minimal, and let $X$ be a quasi-function from $D^2$ to $E$. If $A,B\subseteq D$, we let $X(A,B)$ be the set of $z\in E$ such that for some $x\in A$ and $y\in B$ we have $(x,y,z)\in X$.
\end{notation}

\begin{definition}\label{small expansion} Let $D$ and $E$ be strongly minimal, and let $X$ be a quasi-function from $D^2$ to $E$. We say that $X$ has \textbf{symmetric small expansion} if for all $\epsilon>0$, for arbitrarily large natural numbers $N$, we can find $A\subseteq D$ with $|A|=N$ and $|X(A,A)|<N^{1+\epsilon}$.
\end{definition}

\begin{definition} Let $D$ be strongly minimal. We say that $D$ has \textbf{universal symmetric small expansion} if every quasi-function from $D^2$ to $D$ has symmetric small expansion.
\end{definition}

The following is a restatement of Definition \ref{small expansion}:

\begin{lemma}\label{small expansion second def} Let $D$ and $E$ be strongly minimal, and let $X$ be a quasi-function from $D^2$ to $E$.
Then $X$ has symmetric small expansion if and only if one can find a sequence of natural numbers $\{N_k\}\rightarrow\infty$, and sets $A_k\subseteq D$ for each $k$ with $|A_k|=N_k$, such that $\lim_{k\rightarrow\infty}\frac{\log{|X(A_k,A_k)}|}{N_k}\leq 1$.
\end{lemma}

We will call the sequence $\{A_k\}$ from Definition \ref{small expansion second def} a \textbf{witness family} of the symmetric small expansion of $X$.

The following will also be useful:

\begin{definition} Let $D$ and $E$ be strongly minimal, let $X$ be a quasi-function from $D^2$ to $E$, and let $\sim$ be an equivalence relation on $D$. We say that $X$ has \textbf{symmetric small expansion respecting $\sim$} if one can find a witness family $\{A_k\}$ to symmetric small expansion of $X$, such that each $A_k$ is a union of $\sim$-classes.
\end{definition}

We proceed to note several preservation properties of symmetric small expansion, which are easy to see using the definitions. Lemma \ref{small expansion preservation under almost equality} gives preservation under almost equality. Lemma \ref{small expansion preservation under finite-to-one maps 1}, Lemma \ref{small expansion preservation under finite-to-one maps 2}, and Corollary \ref{universal small expansion preservation} discuss behavior under finite-to-one functions.

\begin{lemma}\label{small expansion preservation under almost equality} Let $D_1$, $D_2$, $E_1$, and $E_2$ be strongly minimal. Assume that $D_1$ and $D_2$ are almost equal, and $E_1$ and $E_2$ are almost equal
\begin{enumerate}
    \item If $X_1$ is a quasi-function from $(D_1)^2$ to $E_1$, then there is a quasi-function $X_2$ from $(D_2)^2$ to $E_2$ that is almost equal to $X_1$.
    \item If $X_1,X_2\subseteq (D_1)^2\times E_1$ are definable and almost equal, then $X_1$ is a quasi-function from $(D_1)^2$ to $E_1$ if and only if $X_2$ is.
    \item If $X_1$ is a quasi-function from $(D_1)^2$ to $E_1$, $X_2$ is a quasi-function from $(D_2)^2$ to $E_2$, and $X_1$ and $X_2$ are almost equal, then $X_1$ has symmetric small expansion if and only if $X_2$ does.
\end{enumerate}
\end{lemma}
\begin{proof}
\begin{enumerate}
    \item It suffices to note that the restriction of $X_1$ to $(D_1\cap D_2)^2\times(E_1\cap E_2)$ has Morley rank 2, which is easy to see using the definition of quasi-functions.
    \item Clear from the definitions.
    \item Assume $X_1$ has symmetric small expansion, and let $\{(N_k,A_k)\}$ be a witness family. It suffices to observe that, after deleting finite subsets of each $A_k$ of bounded size, we obtain sets $A'_k\subseteq D_2$ with $|X_2(A'_k,A'_k)|\leq|X_1(A_k,A_k)|+l\cdot N_k$ for some fixed integer $l$. Clearly, this implies the desired result. Now to see that this is possible, we delete all points of the symmetric difference $D_1\Delta D_2$, as well as all points of $D_1\cup D_2$ which have infinite preimage under the projection of $X_2-X_1$ to the leftmost $D_1\cap D_2$ factor. Since $\dim(X_2-X_1)\leq 1$, there are only finitely many such points, and by definition they do not depend on $k$ -- thus $|A_k-A'_k|$ is indeed bounded.
    
    Now the point is that, outside the removed points, the projection of $X_2-X_1$ to the first copy of $D_1\cap D_2$ is finite-to-one with fibers of size at most some fixed integer $l$. This easily implies that the image of $X_2-X_1$ in $(D_1\cup D_2)^2$ can contain at most $l\cdot N_k$ points of $A'_k\times A'_k$. Thus $X_2(A'_k,A'_k)$ has size at most $|X_1(A_k,A_k)|+l\cdot N_k$, as desired. 
    \end{enumerate}
\end{proof}

\begin{remark} It follows by Lemma \ref{small expansion preservation under almost equality}(3) (by setting $X_1=X_2$) that whether a quasi-function $X$ from $D^2$ to $E$ has symmetric small expansion is purely a property of $X$, and does not depend on the choice of $D$ and $E$.
\end{remark}

\begin{lemma}\label{small expansion preservation under finite-to-one maps 1} Let $D_1$, $D_2$, $E_1$, and $E_2$ be strongly minimal, and let $f:D_1\rightarrow D_2$, $g:E_1\rightarrow E_2$ be definable, finite-to-one, surjective functions.
\begin{enumerate}
    \item If $X_1$ is a quasi-function from $(D_1)^2$ to $E_1$, then $X_2:=\{(f(x),f(y),g(z)):(x,y,z)\in X_1\}$ is a quasi-function from $(D_2)^2$ to $E_2$.
    \item If $X_2$ is a quasi-function from $(D_2)^2$ to $E_2$, then there is a quasi-function $X_1$ from $(D_1)^2$ to $E_1$ such that $X_2=\{(f(x),f(y),g(z)):(x,y,z)\in X_1\}$.
\end{enumerate}
\end{lemma}
\begin{proof}
\begin{enumerate}
    \item Easy by the definitions.
    \item The preimage of $X_2$ in $(D_1)^2\times E_1$ (under $f$, $f$, and $g$) has dimension 2. One can then choose $X_1$ to be any stationary component of this preimage, so long as the projection of $X_1$ to $X_2$ (using $f$, $f$, and $g$) is surjective. This is possible because each of $f$ and $g$ is surjective.
    \end{enumerate}
\end{proof}

\begin{lemma}\label{small expansion preservation under finite-to-one maps 2} Let $D_1$, $D_2$, $E_1$, and $E_2$ be strongly minimal. Let $f:D_1\rightarrow D_2$, $g:E_1\rightarrow E_2$ be definable finite-to-one surjective functions. Let $X_1$ be a quasi-function from $(D_1)^2$ to $E_1$, and let $X_2=\{(f(x),f(y),g(z)):(x,y,z)\in X_1\}$.
\begin{enumerate}
    \item $X_2$ has symmetric small expansion if and only if $X_1$ has symmetric small expansion respecting the equivalence relation $f(x)=f(y)$ on $D_1$.
    \item In particular, if $X_2$ has symmetric small expansion, then so does $X_1$.
    \item If $D_1=D_2$ and $f$ is the identity map, then any witness family for symmetric small expansion of $X_1$ is also a witness family for $X_2$, and vice versa. In particular, for any equivalence relation $\sim$ on $D_1$, $X_1$ has symmetric small expansion respecting $\sim$ if and only if $X_2$ does.
\end{enumerate}
\end{lemma}
\begin{proof}
\begin{enumerate}
    \item Assume $X_1$ has symmetric small expansion respecting the relation $f(x)=f(y)$. Let $\{(N_k,A_k)\}$ be a witness. Then $|f(A_k)|=O(N_k)$. Moreover, it is easy to see (using that the $A_k$ are closed under the relation $f(x)=f(y)$) that $X_2(f(A_k),f(A_k))=g(X_1(A_k,A_k))$, and thus that $|X_2(f(A_k),f(A_k))|\leq l\cdot|X_1(A_k,A_k)|$ for some fixed integer $l$. This shows that $\{f(A_k)\}$ is a witness family for $X_2$.
    
    Now assume $X_2$ has symmetric small expansion, and let $\{A_k\}$ be a witness family. Then by similar reasoning, $\{f^{-1}(A_k)\}$ is a witness family for $X_1$.
    \item By (1).
    \item By the proof of (1).
\end{enumerate}
\end{proof}

\begin{corollary}\label{universal small expansion preservation} Let $D_1$ and $D_2$ be strongly minimal, and let $f:D_1\rightarrow D_2$ be a definable finite-to-one surjective function. If $D_2$ has universal symmetric small expansion, then so does $D_1$.
\end{corollary}
\begin{proof} By Lemmas \ref{small expansion preservation under finite-to-one maps 1}(1) (setting $E_1=D_1$, $E_2=D_2$, $g=f$) and \ref{small expansion preservation under finite-to-one maps 2}(2).
\end{proof}

We turn now toward verifying symmetric small expansion in certain `group-like' environments. Our main tool is the following lemma.

\begin{notation} If $R$ is a ring, $S$ and $T$ are subsets of $R$, and $a,b\in R$, then by $a\cdot S+b\cdot T$ we mean the set of all elements of $R$ of the form $a\cdot s+b\cdot t$ for $s\in S$ and $t\in T$.
\end{notation}

\begin{lemma}\label{small expansion in int dom} Let $R$ be an integral domain of characteristic zero, and let $\tau_1,...,\tau_n\in R$. Then one can find finite sets $S_k\subseteq R$ with $\lim_{k\rightarrow\infty}|S_k|=\infty$, such that $$\lim_{k\rightarrow\infty}\frac{\log{|Z_k|}}{\log{|S_k|}}\leq 1,$$ where $$Z_k=\bigcup_{i,j\leq n}\tau_i\cdot S_k+\tau_j\cdot S_k.$$
\end{lemma}

\begin{proof} The proof is essentially identical to the argument in \cite{JTR} (Lemma 3.1 and Proposition 3.2); we give a sketch. First, note that sets $S_k$ with the desired properties are closed under scaling: that is, if $\{S_k\}$ is any such sequence, then so is $\{c_kS_k\}$ for any nonzero $c_k\in R$. It follows that we may assume $R$ is a field.

Now let us rewrite our given elements as $\sigma_1,...,\sigma_l$, $\tau_1,...,\tau_m$, where the $\sigma_i$ are algebraically independent transcendentals, and each $\tau_j$ has degree $d_j$ over $\mathbb Q(\sigma_1,...,\sigma_l,\tau_1,...,\tau_{j-1})$. For each $d>0$ Let $Q_d$ be the set of monomials in $(x_1,...,x_l,y_1,...,y_m)$ of degree $<d$ in each variable, and let $P_d$ be the set of $q\in Q_d$ of degree $<d_j$ in each $y_j$. Then for each $d,r>0$ let $B_r^d$ be the set of all linear combinations $\sum_{q\in Q_d} a_q\cdot q(\bar\sigma,\bar\tau)$, where each $a_q\in\{0,...,r-1\}$.

We claim that, if we fix $d>\max\{d_1,...,d_m\}$ and consider $|B_r^d|$ as a function of $r$, we obtain $|B_r^d|=O(r^D)$, where $D=d^ld_1...d_m$. To see this, note first that the elements $q(\bar\sigma\bar\tau)$ for $q\in P_d$ are linearly independent over $\mathbb Q$, by definition of the $d_j$. Thus $|B_r^d|\geq r^D$. On the other hand, using the minimal polynomials of each $\tau_j$, one sees that for fixed $d$, all $q(\bar\sigma\bar\tau)$ for $q\in Q_d$ can be expressed in the form $$\frac{\sum_{q\in P_d}a_q\cdot q(\bar\sigma,\bar\tau)}{b},$$ where $b\in R$ is fixed, and the $a_q$ are integers. It follows easily that $|B_r^d|$ is bounded by a constant multiple of $r^D$, as desired.

Next note that, for $\alpha,\beta\in\{\sigma_1,...,\sigma_l,\tau_1,...,\tau_m\}$, we have $\alpha\cdot B_r^d+\beta\cdot B_r^d\subseteq B_{2r}^{d+1}$. Define $Z_r^d$ as the union of all such sets $\alpha\cdot B_r^d+\beta\cdot B_r^d$ for $\alpha,\beta$ among our $n$ given elements. Then we obtain $$\lim_{r\rightarrow\infty}\frac{\log{|Z_r^d|}}{\log{|B_r^d|}}\leq\lim_{r\rightarrow\infty}\frac{\log{|B_{2r}^{d+1}|}}{\log{|B_r^d|}}=\lim_{r\rightarrow\infty}\frac{\log{r^{(d+1)^ld_1...d_m}}}{\log{r^{d^ld_1...d_m}}}=\left(\frac{d+1}{d}\right)^l.$$ Thus, letting also $d\rightarrow\infty$, we can make our desired limit go to 1, and therefore extract a sequence $S_k$ as in the lemma.
\end{proof}

We now turn toward small expansion in strongly minimal groups. Recall that a definable endomorphism of a strongly minimal group is either trivial or surjective; in particular, the ring of definable endomorphisms of a strongly minimal group has no zero divisors.

Now the first case of our desired result is the following proposition.

\begin{definition} Let $G$ be a definable set, and let $D\subseteq G^n$ be strongly minimal. We say that $D$ is \textbf{irredundant in $G$} if each of the $n$ projections $D\rightarrow G$ is finite-to-one.
\end{definition}

\begin{proposition}\label{small expansion in cosets} Let $G$ be a divisible strongly minimal group, whose ring of definable endomorphisms is commutative. Let $D,E\subseteq G^n$ be strongly minimal and irredundant in $G$, and let $X$ be a quasi-function from $D^2$ to $E$. Assume that each of $D$, $E$, and $X$ is a coset of definable subgroup of the appropriate power of $G$, and let $\sim$ be the equivalence relation on $D$ induced by the permutations of the coordinates $\{1,...,n\}$. Then $X$ has symmetric small expansion respecting $\sim$.
\end{proposition}

\begin{proof}
 For each $i,j\leq n$, let $\pi_{ij1}(X)$ be the image of $X$ in $G^3$, given by the applying the $ith$, $jth$, and 1st coordinate projections to the three factors $D$, $D$, and $E$, respectively. Since $D$ and $E$ are irredundant in $G$, it follows that each $\pi_{ij1}(X)$ is a quasi-function from $G^2$ to $G$. Moreover, since $X$ is a coset, so is $\pi_{ij1}(X)$.
 \begin{claim} We may assume that each $\pi_{ij1}$ is the graph of a function $G^2\rightarrow G$.
 \end{claim}
 \begin{proof} Consider the leftmost projection $\pi_{ij1}(X)\rightarrow G^2$. Since $\pi_{ij1}$ is both a quasi-function and a coset, one concludes easily that the fibers of $\pi_{ij1}(X)\rightarrow G^2$ (viewed as subsets of the third $G$ factor) are cosets of a fixed finite subgroup $H_{ij}\leq G$. Let $m$ be a common multiple of all of the $|H_{ij}|$. Then the map $x\mapsto m\cdot x$ is constant on each fiber. So if we replace $E$ with its image under scaling the first coordinate by $m$, the $\pi_{ij1}$ will be replaced by the graph of a function.
 
 Formally, let $E'$ be the image of $E$ under scaling the first coordinate by $m$, and let $g:E\rightarrow E'$ be the map that scales the first coordinate by $m$. Then $E'$ is also a coset, and is also irredundant in $G$. Let $X'$ be the image of $X$ in $D^2\times E'$. By Lemma \ref{small expansion preservation under finite-to-one maps 2}(3), $X$ has small expansion respecting $\sim$ if and only if $X'$ does. So we may replace $E$ with $E'$ and $X$ with $X'$. Finally, it follows from above that each $\pi_{ij1}(X')$ is indeed the graph of a function $G^2\rightarrow G$.
  \end{proof}

 Now assume each $\pi_{ij1}(X)$ is the graph of a function $G^2\rightarrow G$. Since $\pi_{ij1}(X)$ is a coset, this function has the form $(x,y)\mapsto\sigma_{ij}(x)+\tau_{ij}(y)+a_{ij}$ for some $a_{ij}\in G$ and some definable endomorphisms $\sigma_{ij},\tau_{ij}$ of $G$.

 Now by assumption, the ring of definable endomorphisms of $G$ is commutative, thus an integral domain. Let $\{S_k\}$ and $\{Z_k\}$ be as in Lemma \ref{small expansion in int dom} for all of the $\sigma_{ij}$ and $\tau_{ij}$. So the $S_k$ and $W_k$ are sets of definable endomorphisms of $G$. Fix $\hat g\in G$ generic over all relevant data, and for each $k$ let $T_k$ and $W_k$ be the set of images $\hat g$ under all endomorphisms in $S_k$ and $Z_k$, respectively. Note that since $\hat g$ is generic we have $|S_k|=|T_k|$ and $|T_k|=|W_k|$. Now for each $k$ let $A_k$ be the set of $(x_1,...,x_n)\in D$ such that $x_i\in T_k$ for some $i$, and note by definition that each $A_k$ is closed under $\sim$.
 
 We will prove the proposition by showing that $\{A_k\}$ is a witness family for symmetric small expansion on $X$. It is enough to show the following two claims:
 
 \begin{claim} $|A_k|=O(|S_k|)$.
 \end{claim} 
 \begin{proof}
 Since each $|S_k|=|T_k|$, it is equivalent to show that $|A_k|=O(|T_k|)$. For each $i$, the set of $(x_1,...,x_n)\in D$ with $x_i\in T_k$ is the preimage of $T_k$ in $D$ under the $ith$ projection $\pi_i$. Since $D$ is irredundant in $G$, $\pi_i:D\rightarrow G$ is finite-to-one. Since $D$ is a coset, $\pi_i$ is moreover surjective and $l_i$-to-1 for some positive integer $l_i$. Thus $\pi_i^{-1}(T_k)\cap D$ has size $l_i\cdot|T_k|$. Since $A_k$ is the union of the $\pi_i^{-1}(T_k)\cap D$, the claim follows.
 \end{proof}
 
 \begin{claim} $|X(A_k,A_k)| = O(|Z_k|)$.
 \end{claim}
 \begin{proof} For each $k$ let $V_k$ be the union of all translates $W_k+a_{ij}$ for $i,j\leq n$, and let $U_k$ be the preimage of $V_k$ in $E$ under $\pi_1$. Clearly $|U_k|=O(|Z_k|)$. We will show that $X(A_k,A_k)\subseteq U_k$, which proves the claim.
 
 To see this, let $\bar z=(z_1,...,z_n)\in X(A_k,A_k)$. Then there are $\bar x=(x_1,...,x_n),\bar y=(y_1,...,y_n)\in A_k$ such that $(\bar x,\bar y,\bar z)\in X$. By definition, there are $i,j\leq n$ with $x_i,y_j\in T_k$. Then $z_1=\sigma_{ij}(x_i)+\tau_{ij}(y_j)+a_{ij}\in W_k+a_{ij}\subseteq V_k$, and thus $\bar z\in U_k$.
 \end{proof}
 
 Finally, by Lemma \ref{small expansion in int dom} we have $\lim_{k\rightarrow\infty}|S_k|=\infty$ and $\lim_{k\rightarrow\infty}\frac{\log{|Z_k|}}{\log{|S_k|}}\leq 1$. Combined with the above two claims, the same assertions hold of $A_k$ and $X(A_k,A_k)$, which proves the proposition.
 \end{proof}
 
 We now use Proposition \ref{small expansion in cosets} to prove Theorem \ref{small expansion when non-orthogonal to group}, the main result of this section.
 
 \begin{notation} If $X$ is a set, then by $X^{(n)}$ we mean the $n$th symmetric power of $X$, i.e. the quotient of $X^n$ by coordinate permutations.
 \end{notation}
 \begin{definition} If $G$ is strongly minimal, and $D\subseteq G^{(n)}$ is strongly minimal for some $n$, then $D$ is \textbf{irredundant in $G$} if each $g\in G$ belongs to only finitely many of the sets in $D$.
 \end{definition}

  \begin{theorem}\label{small expansion when non-orthogonal to group} Let $G$ be a locally modular, divisible, strongly minimal group, whose ring of definable endomorphisms is commutative. Then every strongly minimal set in finite correspondence with $G$ has universal symmetric small expansion.
 \end{theorem}
 \begin{proof} Let $D$ be strongly minimal and in finite correspondence with $G$. We may view the finite correspondence as a definable, finite-to-one function $D\rightarrow G^{(n)}$ for some $n$. By choosing the smallest possible $n$, one easily ensures that the image of $D$ in $G^{(n)}$ is irredundant in $G$. So by Corollary \ref{universal small expansion preservation}, we may without loss of generality replace $D$ with this image. Thus, from now on, we assume $D$ is itself an irredundant subset of $G^{(n)}$. Let us also fix a strongly minimal set $D'\subseteq G^n$ projecting to $D$.
 
 Now let $X$ be a quasi-function from $D^2$ to $D$. By Lemma \ref{small expansion preservation under finite-to-one maps 1}, there is a quasi-function $X'$ from $(D')^2$ to $D'$ that projects exactly to $X$.
 
 By the classification of definable sets in 1-based groups (\cite{HrPi87}), equivalently Fact \ref{nontrivial locally modular characterization}), there are cosets $D_1$ and $X_1$ which are almost equal to $D'$ and $X'$, respectively. It follows easily that $X_1\subseteq(D_1)^3$, and thus $X_1$ is a quasi-function from $(D_1)^2$ to $D_1$ (by repeated applications of Lemma \ref{small expansion preservation under almost equality}(2) and (3)). Let $D_2$ and $X_2$ be the images of $D_1$ and $X_1$ under $G^n\rightarrow G^{(n)}$. So $D_2$ is almost equal to $D$, $X_2$ is almost equal to $X$, and $X_2$ is a quasi-function from $(D_2)^2$ to $D_2$.
 
 Since $D$ is irredundant in $G$, so is $D'$, and thus so is $D_1$. Then by Proposition \ref{small expansion in cosets}, $X_1$ has symmetric small expansion respecting coordinate permutations. Then by Lemma \ref{small expansion preservation under finite-to-one maps 2}, $X_2$ has symmetric small expansion. So by Lemma \ref{small expansion preservation under almost equality}, so does $X$.
 \end{proof}

\section{Characterizing Locally Modularity}

Throughout Section 6, we let $\mathscr M$ be a non-trivial locally modular reduct of $(\mathbb C,+,\times)$. We will show that every $\mathscr M$-definable polynomial $P(x,y)\in\mathbb C[x,y]-\mathbb C[x]-\mathbb C[y]$ is either linear or a twisted monomial. We conclude that if $\mathscr M$ is a polynomial reduct, it must be interdefinable with a vector space or a twisted multiplication. 

Lemma \ref{L: summary of small expansion} gives the necessary input from Section 5:

\begin{lemma}\label{L: summary of small expansion} Let $P(x,y)\in\mathbb C[x,y]-\mathbb C[x]-\mathbb C[y]$ be $\mathscr M$-definable. Then $P$ has symmetric small expansion.
\end{lemma}
\begin{proof} By Fact \ref{nontrivial locally modular characterization}, there is a locally modular, strongly minimal, abelian group $G$ in definable finite correspondence with $M$. The lemma then follows from Theorem \ref{small expansion when non-orthogonal to group}, provided one can show $G$ is divisible with commutative definable endomorphism ring. 

Now by \cite{vdDriesGpChunk} (Theorem 3), we may assume $G$ is an algebraic group, which (by the finite correspondence with $\mathbb C$) has dimension 1. So the connected component of the identity in $G$ is isomorphic (as an algebraic group) to either $(\mathbb C,+)$, $(\mathbb C^\times,\times)$, or an elliptic curve. In any case, $G$ has unbounded exponent. By strong minimality (in the structure $\mathscr M$), it follows easily that $G$ is divisible.

Now as a divisible algebraic group, $G$ is irreducible. Thus $G$ itself is a copy of the additive group, the multiplicative group, or an elliptic curve. In any case, the definable endomorphism ring is then commutative. Indeed, every definable endomorphism is an algebraic endomorphism, and the algebraic endomorphism rings of these groups are commutative (they are all isomorphic to either $\mathbb C$, $\mathbb Z$, or $\mathbb Z^2$). 
\end{proof}

We now apply tools from additive combinatorics to get more information about $\mathscr M$-definable polynomials. The following ad hoc terminology will be useful. 

\begin{definition}\label{D: additive multiplicative} Let $P(x,y)\in\mathbb C[x,y]-\mathbb C[x]-\mathbb C[y]$. 
\begin{enumerate}
    \item $P$ is \textbf{weakly additive} if there are non-constant unary polynomials $f,u,v$ such that $P(x,y)=f(u(x)+v(y))$.
    \item $P$ is \textbf{weakly multiplicative} if there are non-constant unary polynomials $f,u,v$ such that $P(x,y)=f(u(x)v(y))$.
    \item $P$ is \textbf{strongly additive} if there are non-constant unary polynomials $f,u$, and constants $c_1,c_2$, such that $P(x,y)=f(c_1u(x)+c_2u(y))$.
    \item $P$ is \textbf{strongly multiplicative} if there are non-constant unary polynomials $f,u$, and integers $m,n\geq 0$, such that $P(x,y)=f(u^m(x)u^n(y))$.
    \end{enumerate}
    \end{definition}
    
By Lemma \ref{L: summary of small expansion} and the main result of~\cite{JTR}, we have:

\begin{fact}\label{C: application of Minh}
Let $P(x,y)\in\mathbb C[x,y]-\mathbb C[x]-\mathbb C[y]$. Then $P$ is either strongly additive or strongly multiplicative.
\end{fact}

Given an $\mathscr M$-definable polynomial $P(x,y)$, our remaining task is to determine the possibilities for the polynomials $f$ and $u$ in the definition of strongly additive (resp. multiplicative). We do this by defining new polynomials in terms of $P$, and using that they also satisfy Fact~\ref{C: application of Minh}. Lemma \ref{L: additive and multiplicative are exclusive} and Corollary \ref{C: weak implies strong} will tell us that these new polynomials remain in the same case as $P$ (additive or multiplicative).

\begin{lemma}\label{L: additive and multiplicative are exclusive} Suppose $P(x,y)\in\mathbb C[x,y]-\mathbb C[x]-\mathbb C[y]$.
\begin{enumerate}
    \item If $P(x,y)=f(u(x)+v(y))$ for some unary $f,u,v$ then $P$ is non-constant on every horizontal and vertical line.
    \item If $P(x,y)=f(u(x)v(y))$ for some unary $f,u,v$ then $P$ is constant on the line $x=a$ if and only if $u(a)=0$, and is constant on the line $y=b$ if and only if $v(b)=0$. In particular, $P$ is constant on at least one horizontal line and at least one vertical line.
    \item $P$ is not both weakly additive and weakly multiplicative.
\end{enumerate}
\end{lemma}
\begin{proof} (1) and (2) are obvious, and (3) is immediate from (1) and (2).
\end{proof}

\begin{corollary}\label{C: weak implies strong} Let $P(x,y)\in\mathbb C[x,y]-\mathbb C[x]-\mathbb C[y]$ be $\mathscr M$-definable. Then: 
\end{corollary}
\begin{enumerate}
    \item $P$ is strongly additive if and only if it is weakly additive, and is strongly multiplicative if and only if it is weakly multiplicative.
    \item If $P$ is weakly additive then $\deg_x(P)=\deg_y(P)$.
    \item If $P$ has the form $g(v(x)w(y))$ for some unary $g,v,w$, then $v$ and $w$ have the same roots.
\end{enumerate}

\begin{proof} (1) follows from Fact~\ref{C: application of Minh} and Lemma \ref{L: additive and multiplicative are exclusive}. (2) follows from (1) and the definition of strongly additive. For (3), it follows from Lemma \ref{L: additive and multiplicative are exclusive}(2) that the roots of $v$ and $w$ are intrinsic to $P$, and do not depend on the choice of $g$, $v$, and $w$. But by (1), $P$ is strongly multiplicative. So we may assume that $v=u^m$ and $w=u^n$ for some $u$, making the conclusion obvious.
\end{proof}

We can now classify the additive case:

\begin{proposition}\label{P: additive case} Suppose $P(x,y)\in\mathbb C[x,y]-\mathbb C[x]-\mathbb C[y]$ is $\mathscr M$-definable and strongly additive. Then $P$ is linear.
\end{proposition}
\begin{proof} Write $P(x,y)=f(c_1u(x)+c_2u(y))$ as in Definition \ref{D: additive multiplicative}.
\begin{claim} Every $\mathscr M$-definable unary polynomial is linear.
\end{claim}
\begin{proof} Let $g$ be $\mathscr M$-definable and unary. We may assume $g$ is not constant. Then $$Q(x,y)=P(x,g(y))=f(c_1u(x)+c_2u(g(y)))$$ is $\mathscr M$-definable and weakly additive. By Corollary \ref{C: weak implies strong}, $\deg_x(Q)=\deg_y(Q)$. So $\deg(u)=\deg(u\circ g)$, thus $\deg(g)=1$, thus $g$ is linear.
\end{proof}
Now to prove Proposition \ref{P: additive case}, note that $g(x)=f(c_1u(x))$ is $\mathscr M$-definable, by choosing $y\in\mathbb C$ with $u(y)=0$. By the claim, $f(c_1u(x))$ is linear. This implies that $f$ and $u$ are linear, and thus so is $P$.
\end{proof}

The multiplicative case is similar but trickier. We will need the following:

\begin{lemma}\label{L: algebra exercise} Suppose $f$ and $g$ are non-constant unary polynomials over $\mathbb C$, and $f$ and $f\circ g$ have the same roots.
\begin{enumerate}
    \item If $g$ is not linear, then $f$ has only one root.
    \item If $r$ is the unique root of $f$, then $g$ is a monomial twisted by $r$.
\end{enumerate}
\end{lemma}
\begin{proof} Let $R$ be the (finite) set of roots of $f$. A restatement of the assumption on $f$ and $g$ is that $g(R),g^{-1}(R)\subseteq R$. Using that $g$ is surjective and $R$ is finite, one concludes easily that $g(R)=g^{-1}(R)=R$. So $g$ restricts to a permutation of $R$.

Now for each $r\in R$, we have $s=g(r)\in R$. Since $g^{-1}(R)=R$, $r$ is the unique solution of $g(x)=s$. Let $d=\deg(g)$. Then $g-s$ has a root of multiplicity $d$ at $x=r$, so $g(x)=a(x-r)^d+s$ for some constant $a$. We now show (1) and (2):
\begin{enumerate}
    \item Let $r\in R$ be arbitrary, and write $g(x)=a(x-r)^d+s$ as above, where we have $d\geq 2$ by assumption. Then $g'(x)=ad(x-r)^{d-1}$ is non-constant, so has $r$ as its unique root. Thus $r$ is determined by $g$, which shows (1).
    \item Let $r$ be the unique root, and write $g(x)=a(x-r)^d+s$ as above. The uniqueness of $r$ implies $s=r$, so $g(x)=a(x-r)^d+r$, as desired.
    \end{enumerate}
    \end{proof}

We now prove the multiplicative case, using the same idea as in the additive case:

\begin{proposition}\label{P: multiplicative case} Suppose $P(x,y)\in\mathbb C[x,y]-\mathbb C[x]-\mathbb C[y]$ is $\mathscr M$-definable and strongly multiplicative. Then $P$ is a twisted monomial.
\end{proposition}

\begin{proof} Write $P(x,y)=f(u^m(x)u^n(y))$ as in Definition \ref{D: additive multiplicative}.
\begin{claim} $u$ has only one root, say $r$, and every $\mathscr M$-definable unary polynomial is a monomial twisted by $r$. 
\end{claim}
\begin{proof} If $g(x)$ is any $\mathscr M$-definable non-constant unary polynomial, then $$Q(x,y)=P(x,g(y))=f(u^m(x)u^n(g(y)))$$ is $\mathscr M$-definable and weakly multiplicative. So by Lemma \ref{C: weak implies strong}, $u^m$ and $u^n\circ g$ have the same roots, thus $u$ and $u\circ g$ have the same roots. 

Now since we are in the multiplicative case, $f(x,x)$ is not linear. So by Lemma \ref{L: algebra exercise}(1) with $g(x)=f(x,x)$, $u$ has a unique root. Now apply Lemma \ref{L: algebra exercise}(2).
\end{proof}

Without loss of generality, we will assume the unique root of $u$ is $0$. Then $u$ is a monomial, and we want to prove that $P$ is a monomial.

To do this, note that $g(x)=f(u^m(x))$ is $\mathscr M$-definable, by choosing $y\in\mathbb C$ with $u(y)=1$. By the claim, $f(u^m(x))$ is a monomial. But now $u$ and $f\circ u^m$ are both monomials, which clearly implies that $f$ is a monomial. Then $P(x,y)=f(u^m(x)u^n(x))$ where $f$ and $u$ are monomials, which shows that $P$ is a monomial.
\end{proof}

Finally, we complete our analysis of the locally modular case:

\begin{theorem}\label{loc mod classification} Let $\mathcal P$ be a collection of complex polynomials, and let $\mathscr M=(\mathbb C;\mathcal P)$. Assume that $\mathscr M$ is non-trivial and locally modular. Then either every $P\in\mathcal P$ is linear, or every $P\in\mathcal P$ is a monomial twisted by the same $r$. In particular, $\mathscr M$ is interdefinable with a vector space or a twisted multiplication.
\end{theorem}
\begin{proof}
By Lemma \ref{unary implies trivial}, there is $P\in\mathcal P$ involving at least two variables. Specializing all other variables generically, we may assume $P$ is binary. By Fact~\ref{C: application of Minh}, Proposition \ref{P: additive case}, and Proposition \ref{P: multiplicative case}, $P$ is either linear or a twisted monomial. Then by Proposition \ref{P: loc mod examples}, one of the group operations $+$ or $\times_r$ (for some $r$) is $\mathscr M$-definable. Let $\otimes$ be this group operation. Then by the classification of definable sets in 1-based groups (\cite{HrPi87}), every $\mathscr M$-definable polynomial is of the form $\bar x\mapsto\bar a\otimes f(\bar x)$, where $f$ is a $\otimes$-homomorphism. If $\otimes=+$, all such polynomials are linear. If $\otimes=\times_r$, all such polynomials are monomials twisted by $r$. Now apply Proposition \ref{P: loc mod examples}.  
\end{proof}

\section{The Non-Locally Modular Case}

Suppose $\mathscr M=(M;...)$ is a non-locally modular strongly minimal structure interpreted in an algebraically closed field. In this section, we give a new condition implying that $\mathscr M$ defines all constructible subsets of all powers of $M$. It is straightforward to see that polynomial reducts satisfy this condition, so that non-locally modular polynomial reducts are interdefinable with the full field structure. The proof of this implication will be given in Section 8. We  assume knowledge of standard definitions and results from algebraic geometry; see~\cite{Shafarevich}, for example.

\begin{definition} Let $K$ be an algebraically closed field, let $V$ and $W$ be smooth, pure-dimensional varieties over $K$, and let $f:V\rightarrow W$ be a definable partial function.
\begin{itemize}
\item We say that $f$ is \textbf{generically defined} if its domain is dense in $V$.
\item We say that $f$ is \textbf{almost regular} at $x\in V$ if there is a neighborhood $V'\subseteq V$ of $x$ such that $f$ is defined on $V'$, and the graph of $f$ restricts to a relatively closed, pure-dimensional subset of $V'\times W$. 
\end{itemize}
\end{definition}

\begin{remark}\label{closed graph} Note that a generically defined map is almost regular at every generic point (where genericity is in the sense of the structure $K$). Moreover, it follows easily from the smoothness of the variety involved that almost regularity is closed under composition: if $f:V\rightarrow W$ is almost regular at $x\in V$, and $g:W\rightarrow Z$ is almost regular at $f(x)\in W$, then $g\circ f$ if almost regular at $x$. We need almost regularity instead of regularity to allow for the inverse Frobenius map in positive characteristic.
\end{remark}

\begin{theorem}\label{nlm case} Let $(K,+,\times)$ be an algebraically closed field. Let $M$ be a one-dimensional constructible set over $K$, and let $\mathscr M=(M,...)$ be a non-locally modular strongly minimal reduct of the full $K$-induced structure on $M$. Assume there are $\mathscr M$-definable generically defined partial functions $P:M^2\rightarrow M$ and $U:M\rightarrow M$ such that:
\begin{enumerate}
    \item $P$ generically depends on both variables (i.e. $P$ does not, on any generic subset of $M^2$, factor through either projection to $M$).
    \item $U$ is finite-to-one and not generically injective.
\end{enumerate}
Then $\mathscr M$ defines every constructible subset of every power of $M$.
\end{theorem}

\begin{remark}\label{non-injective unary} From the proof, one can weaken condition (2) to a more technical restriction on $U$: namely that whenever $a,b\in M$ are generic and $n$ is a positive integer, the $n$-fold iterate of $x\mapsto P(P(U(x),a),b)$ is not the identity infinitely often. We did not do this because condition (2) above is more natural.
\end{remark}

\begin{remark} Condition (2) is necessary. For example, in positive characteristic, Marker and Pillay ~\cite{MarkerPillay} give an example of a non-locally modular expansion of the additive group that does not define multiplication (so here one takes $P(x,y)=x+y$). We expect that (2) can be dropped in characteristic zero.
\end{remark}

\begin{proof} We may assume $K$ and $\mathscr M$ are saturated, and $P$ and $U$ are $\emptyset$-definable in $\mathscr M$. By editing finitely many points, we may also assume that $M$ is the set of $K$-points of a smooth curve over $K$. By Zilber's Restricted Trichotomy (see \cite{AssafDmitry},\cite{Ben}), there is an $\mathscr M$-definable algebraically closed field $(F,\oplus,\otimes)$ in $\mathscr M$-definable finite correspondence with $M$. By well-known techniques, it will suffice to find an $\mathscr M$-definable injection $M\rightarrow F^n$ for some $n$ (this is proven explicitly in \cite{AssafBen}, Theorem 4.12). 

By elimination of imaginaries in $F$, we can view the given finite correspondence as a finite-to-one $\mathscr M$-definable function $f:M\rightarrow F^n$ for some $n$. Among all such functions (over all $n$), fix $f$ and $n$ so that the generic fiber size of $f$ is minimized. Call this fiber size $k$. So all but finitely many fibers of $f$ have size $k$. 

We will prove Theorem \ref{nlm case} by showing that $k=1$. To do this, we use the provided functions $P$ and $U$ to try to build another function with smaller fibers. Assuming this doesn't work, we will collect increasingly restrictive information about the definable sets in $\mathscr M$. 
 
 \begin{definition} For $x,y\in\mathbb C$, we define $x\sim y$ if and only if $f(x)=f(y)$.
 \end{definition}
 
 So $\sim$ is an $\mathscr M$-definable equivalence relation on $M$, and cofinitely many of the $\sim$-classes have size $k$. 
 
 Our first goal is Lemma \ref{plane curves always respect equivalence relation}, which says that definable finite correspondences on $M$ must respect $\sim$ in a strong way. We prove the lemma by successively generalizing restricted versions.
 
 \begin{lemma}\label{functions preserve equivalence relation} Let $g:M\rightarrow F^m$ be any finite-to-one $\mathscr M$-definable function. Then $g$ generically respects $\sim$: for cofinitely many $x\in M$, all $x'\sim x$ satisfy $g(x)=g(x')$.
 \end{lemma}
 
 \begin{proof} Consider the concatenation map $(f,g):M\rightarrow F^{n+m}$ given by $x\mapsto(f(x),g(x))$. The fibers of $(f,g)$ are intersections of fibers of $f$ and $g$. Fix $x\in M$ generic over the parameters defining these two maps. Then the fiber containing $x$ in $(f,g)$ is a subset of $f^{-1}(f(x))$. By the minimality of $k$, these two fibers are equal. Thus, every $x'\sim x$ satisfies $g(x)=g(x')$. Since $x$ is generic, this property then holds on a cofinite subset of $M$.
 \end{proof}
 
 Lemma \ref{plane curves generically respect equivalence relation} generalizes Lemma \ref{functions preserve equivalence relation} from functions to finite correspondences:
 
 \begin{lemma}\label{plane curves generically respect equivalence relation} Let $C\subseteq M^2$ be any $\mathscr M$-definable finite correspondence. Then for cofinitely many $(x,y)\in C$, for every $x'\sim x$, there is some $y'\sim y$ with $(x',y')\in C$.
 \end{lemma}
 
 \begin{proof} We view $C$ as a multivalued function from $M$ to $M$, say $x\mapsto C(x)\subseteq M$. The composition $f\circ C$ is a multivalued function from $M$ to $F$. Let $X=f\circ C(M)$. By elimination of imaginaries in $F$, there is an $\mathscr M$-definable injection $g:X\rightarrow F^m$ for some $m$.
 
 We now apply Lemma \ref{functions preserve equivalence relation} to the finite-to-one $\mathscr M$-definable function $h=g\circ f\circ C:M\rightarrow F^m$. Namely, let $(x,y)\in C$ be generic over all relevant parameters, and let $x'\sim x$. Note that $x$ is generic in $M$ over the same parameters, since $C$ is a finite correspondence. Now by Lemma \ref{functions preserve equivalence relation}, we have $g\circ f\circ C(x)=g\circ f\circ C(x')$. Since $g$ is injective, $f\circ C(x)=f\circ C(x')$. But $y\in C(x)$, so $f(y)\in f\circ C(x)$, thus $f(y)\in f\circ C(x')$. So there is some $y'\in C(x')$ such that $f(y')=f(y)$, i.e. $y'\sim y$.
 
 So the desired property holds for all generic $(x,y)\in C$ over the relevant parameters, and thus holds on a cofinite subset of $C$.
 \end{proof}
 
 We now arrive at Lemma \ref{plane curves always respect equivalence relation}. Essentially, it says that the finite exceptional set of Lemma \ref{plane curves generically respect equivalence relation} can be traced to certain geometric abnormalities in the set $C$.
 
 \begin{lemma}\label{plane curves always respect equivalence relation} There is a fixed finite set $T\subseteq M$ with the following property: suppose $C\subseteq M^2$ is $\mathscr M$-definable and one-dimensional. Let $X$ and $Y$ be $\sim$-classes, and let $(x_0,y_0)\in C\cap(X\times Y)$ with $x_0,y_0\notin T$. Moreover assume:
 \begin{enumerate}
     \item $C$ is relatively closed and pure-dimensional in a neighborhood of $X\times Y$.
     \item The sets $\{y:(x_0,y)\in C\}$ and $\{x:(x,y_0)\in C\}$ are finite.
 \end{enumerate}
 Then the projections $C\cap(X\times Y)\rightarrow X$ and $C\cap(X\times Y)\rightarrow Y$ are surjective -- that is, for any $x\sim x_0$, there is $y\sim y_0$ with $(x,y)\in C$, and vice versa.
 \end{lemma}
 \begin{proof}
 By (1) and (2) above, we may assume $C$ is pure-dimensional and relatively closed in $M^2$, and that both projections $C\rightarrow M^2$ are finite-to-one. Indeed, one can $\mathscr M$-definably (over parameters) delete all isolated points of $C$, as well as all infinite fibers in each projection $C\rightarrow M^2$, then take the closure of the remaining set. It follows from (1), (2), and (3) that the resulting set is still one-dimensional, $(x_0,y_0)$ was not deleted, and some neighborhood of $X\times Y$ had no new points added. 
 
 So assume $C$ is a closed, pure-dimensional finite correspondence on $M$. Our strategy is to `projectivize' the entire problem. We show (Claim \ref{V to W surjective}) that the projective version `completes' Lemma \ref{plane curves generically respect equivalence relation}, making the desired property hold at \textit{all} points of $C$ instead of \textit{most} points. One then needs to check where exactly any `projective' points were used. We will see that, because $C$ is closed and pure-dimensional in $M^2$, projective points only arise because of $\sim$, and have nothing to do with $C$.

 Let us continue with the proof. Let $\hat M$ be a smooth projective completion of $M$. Let $E\subseteq M^2$ be the graph of $\sim$. So $\dim(E)=1$. Let $\hat C,\hat E\subseteq\hat M^2$ be pure-dimensional projective curves differing from $C$ and $E$ in finitely many points, respectively. So $\hat C$ and $\hat E$ are finite correspondences on $\hat M$. The assumption that $C$ is closed and pure-dimensional in $M^2$, equivalently stated, gives that $C=\hat C\cap M^2$ and $\hat C$ is pure-dimensional.
 
 \begin{definition} Let $T$ be the set of all elements of $M$ which appear as coordinates of elements of the symmetric difference $E\Delta\hat E$.
 \end{definition}
 
 Clearly $T$ is finite and independent of $C$. We will show that $T$ satisfies the lemma. 

 \begin{definition} Define the sets $V\subseteq\hat M^4$ and $W\subseteq\hat M^3$ as follows:
 \begin{itemize}
     \item $(x,y,x',y')\in V$ if $(x,y),(x',y')\in\hat C$ and $(x,x'),(y,y')\in\hat E$.
     \item $(x,y,x')\in W$ if $(x,y)\in\hat C$ and $(x,x')\in\hat E$.
 \end{itemize}
 \end{definition}

 So $V$ and $W$ are projective varieties, and there are natural projections $V\rightarrow W$, $W\rightarrow\hat C$, and $W\rightarrow\hat E$, all of which are finite-to-one (because $\hat C$ and $\hat E$ are finite correspondences). We need the following two additional claims about $V$ and $W$:
 \begin{claim}\label{dim W=1} $W$ is a pure-dimensional curve.
 \end{claim}
 \begin{proof} Since $W\rightarrow\hat C$ is finite-to-one, $\dim(W)\leq 1$. On the other hand, one checks easily that every generic $(x,y)\in C$ is in the image of $W\rightarrow\hat C$, which forces $\dim(W)=1$. Pure-dimensionality then follows because $W$ has codimension 2 in $\hat M^3$ and is defined by the two pure codimension 1 conditions $(x,y)\in\hat C$ and $(x,x')\in\hat E$ (here we use that $\hat M$ is smooth and $C$ is pure-dimensional).
 \end{proof}

 \begin{claim}\label{V to W surjective} $V\rightarrow W$ is surjective.
 \end{claim}
 \begin{proof} 
 By projectivity and Claim \ref{dim W=1}, it suffices to show that $V$ is \textit{almost} surjective, i.e. the image of $V$ is almost equal to $W$. To do this, let $(x,y,x')\in W$ be generic over all relevant parameters. It follows easily from Claim \ref{dim W=1} that $(x,y)$ and $(x,x')$ are generics in $\hat C$ and $\hat E$, respectively, and thus are also generics in $C$ and $E$, respectively. By Lemma \ref{plane curves generically respect equivalence relation}, there is some $y$ with $(y,y')\in E$ and $(x',y')\in C$. Note that $(y,y')$ is also generic in $E$, thus $(y,y')\in\hat E$. So $(x,y,x',y')\in V$, as desired.
 \end{proof}

 We now prove Lemma \ref{plane curves always respect equivalence relation}. Recall that we are given $\sim$-classes $X$ and $Y$, and some $(x_0,y_0)\in C\cap(X\times Y)$ with $x_0,y_0\notin T$. We want to show that $C\cap(X\times Y)\rightarrow X$ and $C\cap(X\times Y)\rightarrow Y$ are surjective. By symmetry, it is enough to consider the projection to $X$. So, let $x\in X$ be arbitrary. So $x_0\sim x$. Since $x_0\notin T$, we get $(x_0,x)\in\hat E$, thus $(x_0,y_0,x)\in W$. By Claim \ref{V to W surjective}, there is $y\in\hat M$ with $(x,y)\in\hat C$ and $(y_0,y)\in\hat E$. Since $y_0\notin T$, we get $(y_0,y)\in E$. In particular, $y\in M$. Then since $C=\hat C\cap M^2$, we get $(x,y)\in C$. 
  \end{proof}
 
 Fix $T$ as in Lemma \ref{plane curves always respect equivalence relation}. Adding finitely many points if necessary, we assume that $T$ is a union of $\sim$-classes, including all classes of size $\neq k$.
 
 \begin{definition} We refer to $\sim$-classes outside $T$ as \textbf{good}.
 \end{definition}

 \begin{convention}
     For the rest of the proof of Theorem \ref{nlm case}, we will absorb $\sim$ and $T$ into the language of $\mathscr M$, thereby assuming they are $\emptyset$-definable. 
     \end{convention}

Corollary \ref{polynomials biject classes} strengthens Lemma \ref{plane curves always respect equivalence relation} in the case of graphs of functions.

 \begin{corollary}\label{polynomials biject classes}
 Let $q:M\rightarrow M$ be a finite-to-one, generically defined, $\mathscr M$-definable partial function. Let $X$ and $Y$ be good $\sim$-classes, and assume $q$ is almost regular at all points of $X$. If $q(x)\in Y$ for some $x\in X$, then $q$ defines a bijection from $X$ to $Y$.
 \end{corollary}
 \begin{proof} Let $C$ be the graph of $q$. By assumption, the restriction of $C$ to some neighborhood of $X\times Y$ is a closed, pure dimensional finite correspondence (see Remark \ref{closed graph}). So Lemma \ref{plane curves always respect equivalence relation} applies. We conclude that $C\cap(X\times Y)\rightarrow Y$ is surjective, and thus $q(X)\supseteq Y$. But $|X|=|Y|=k$ by goodness, so in fact $q$ must biject $X$ and $Y$.
 \end{proof}

\begin{remark}\label{poly bijection at generic point} Let $q:M\rightarrow M$ be as in the setup of Corollary \ref{polynomials biject classes}. Suppose $X\in M/\sim$ is generic in the sense of the structure $(K,+,\times)$, over the parameters defining $M$, $\sim$, $T$, and $q$. Then $q$ is automatically almost regular at all points of $X$. In particular, Corollary \ref{polynomials biject classes} implies that $q$ bijects $X$ with some (necessarily also $K$-generic) $Y\in M/\sim$.
\end{remark}

In light of Remark \ref{poly bijection at generic point}, we make the following harmless convention, which will only matter for applications of Corollary \ref{polynomials biject classes} in the proof of Lemma \ref{composition with two fixed points}:

\begin{convention} For the rest of the proof of Theorem \ref{nlm case}, the term \textbf{generic}, applied to elements of $\sM$-definable sets, always refers to genericity in the structure $(K,+,\times)$, where we absorb the parameters defining $M$ and all basic relations of $\sM$. \end{convention}
 
 Recall that to prove Theorem \ref{nlm case}, we want to show $k=1$. We will do this by combining several iterated applications of Lemma \ref{plane curves always respect equivalence relation} and Corollary \ref{polynomials biject classes}. 
 
 Our first task is to construct a generically defined partial function $Q:M^2\rightarrow M$ such that the unary maps $x\mapsto Q_y(x)=Q(x,y)$ admit a certain configuration of `fixed points' (i.e. solutions to $Q(x,y)=x$). We then conclude by applying Lemma \ref{plane curves always respect equivalence relation} to the set defined by $Q(x,y)=x$. The function $Q$ will be built from $P$ by replacing $x$ with a composition of several unary maps, and applying Corollary \ref{polynomials biject classes} at each step in the composition. It is here that we use our unary function $U$ from the statement of Theorem \ref{nlm case}.
 
 \begin{notation} Given $y\in M$, by $P_y$ we mean the map $x\mapsto P(x,y)$.
 \end{notation}

 \begin{remark}\label{1 dim family} Since $P$ generically depends on both variables, it follows easily that $P_y$ is finite-to-one for generic $y\in M$. It also follows that if $(x,y)\in M^2$ is generic, then so are $(x,P(x,y))$ and $(y,P(x,y))$. Since $M^2$ is stationary, it then follows that if $(x,z)\in M^2$ is generic, there is some generic $y\in M$ such that $P(x,y)=z$ and each of $(x,y)$ and $(y,z)$ is generic in $M^2$.
 \end{remark}
 
 \begin{lemma}\label{composition with two fixed points} There are generic $\sim$-classes $X$ and $Y$, an element $y_0\in Y$, and an $\mathscr M$-definable generically defined partial map $Q:M^2\rightarrow M$, generically depending on both variables, with the following properties: 
 \begin{enumerate}
     \item $Q$ is almost regular at all points of $X\times Y$. In particular, the graph of $Q$ is relatively closed and pure two-dimensional in a neighborhood of $X\times Y\times M$.
     \item For each $x\in X$, map $y\mapsto Q(x,y)$ is finite-to-one and almost regular at all points of $Y$.
     \item The map $x\mapsto Q(x,y_0)$ is the identity on $X$, but is not generically the identity on $M$. 
 \end{enumerate}
 \end{lemma}
 \begin{proof} Let $X_0$ be any generic $\sim$-class, and let $x_0\in X_0$. By Corollary \ref{polynomials biject classes}, the map $U$ bijects $X_0$ with a generic $\sim$-class $X_1$. Since $U$ is $\emptyset$-definable in $\mathscr M$, $X_0$ and $X_1$ are interalgerbraic.
 
 Now let $a\in M$ be generic over $(X_0,X_1)$. By Corollary \ref{polynomials biject classes} and Remark \ref{1 dim family}, the map $P_a$ bijects $X_1$ with a generic $\sim$-class $X_2$ that is independent from $X_1$, and thus also from $X_0$. Let $x_2=P_a\circ U(x_0)\in X_1$. Then $(x_0,x_2)$ is generic in $M^2$, and thus so is $(x_2,x_0)$. So by Remark \ref{1 dim family} again, there is a generic $b\in M$ such that $P_b(x_2)=x_0$ and $x_2$ is generic in $M$ over $b$. By Corollary \ref{polynomials biject classes} again, $P_b$ bijects $X_2$ to $X_0$. 
 
 Now by construction, the map $P_b\circ P_a\circ U$ defines a permutation $\tau$ of the finite set $X_0$. Let $m$ be the order of $\tau$. Then the $m$-fold composition $(P_b\circ P_a\circ U)^m$ induces the identity on $X_0$.
 
 We then satisfy the lemma by setting $X=X_0$, $y_0=b$ (so $Y$ is the class of $b$), and $Q(x,y)=P_y\circ r(x)=P(r(x),y)$, where $$r(x)=P_a\circ U\circ(P_b\circ P_a\circ U)^{m-1}.$$ Let us verify (1)-(3) in the lemma statement:
 \begin{enumerate}
     \item By the construction of $Q$, and since almost regularity is closed under composition, it is enough to see that $U$ is almost regular at all points of $X_0$, $P_a$ is almost regular at all points of $X_1$, $P_b$ is almost regular at all points of $X_2$, and $P$ is almost regular at all points of $X_2\times Y$. In each case, this follows since the relevant points are generic in the domain over the parameters defining the function.
     \item Let $x\in X$, and let $Q^x$ be the map $y\mapsto Q(x,y)$. It follows from (1) that $Q^x$ is almost regular at all points of $y$. We show that $Q^x$ is finite-to-one. For this, note that $r(x)\in X_2$ by construction. So $r(x)$ is generic in $M$, which shows that $Q^x$ (which is $y\mapsto P(r(x),y)$) is a generic specialization of $P$. Now use that $P$ generically depends on both variables (see Remark \ref{1 dim family}).
     \item The map $y_0\mapsto Q(x,y_0)$ is the composition $(P_b\circ P_a\circ U)^m$, which is the identity on $X_0$ by construction. Now suppose toward a contradiction that $y_0\mapsto Q(x,y_0)$ is the identity. Then $y_0\mapsto Q(x,y_0)$ is generically injective, and thus each of $P_b$, $P_a$, and $U$ is generically injective. But $U$ is non-generically injective by assumption, contradiction.
 \end{enumerate}
 \end{proof}

\begin{notation} For the rest of the proof of Theorem \ref{nlm case}, fix $X$, $Y$, $y_0$, and $Q$ as in Lemma \ref{composition with two fixed points}. We also define $C=\{(x,y)\in M^2:Q(x,y)=x\}$.
\end{notation}

By construction, we have $(x,y_0)\in C$ for each $x\in X$. We want to apply Lemma \ref{plane curves always respect equivalence relation} to $C$ at these points, but we first need to check that the hypotheses are satisfied:

\begin{lemma} For each $x\in X$, the hypotheses of Lemma \ref{plane curves always respect equivalence relation} apply to $C$ at $(x,y_0)$. That is:
\begin{enumerate}
    \item $\dim(C)=1$.
    \item $C$ is relatively closed and pure-dimensional in some neighborhood of $X\times Y$.
    \item The set $\{y':(x,y')\in C\}$ is finite.
    \item The set $\{x':(x',y_0)\in C\}$ is finite.
\end{enumerate}
\end{lemma}
\begin{proof} (3) follows from Lemma \ref{composition with two fixed points}(2), and (4) follows from Lemma \ref{composition with two fixed points}(3). That $C$ is closed in a neighborhood of $X\times Y$ follows from Lemma \ref{composition with two fixed points}(1).

Next, note that $C$ can be thought of as the intersection of the graph of $Q$ with the hypersurface $\{(x,y,z):x=z\}$. By Lemma \ref{composition with two fixed points}(1), the graph of $Q$ is pure 2-dimensional in a neighborhood of $X\times Y\times M$, so its intersection with a hypersurface has no isolated points. In particular, since $(x,y_0)\in C$, we get $\dim(C)\geq 1$. 

Now if $\dim(C)=2$, then $C$ is dense in $M^2$. But since $C$ is closed in a neighborhood of $X\times Y$, this would imply that $C$ contains a neighborhood of $X\times Y$, contradicting (3). So $\dim(C)=1$, and we get (1). Finally, since $C$ is closed and has no isolated points in a neighborhood of $X\times Y$, we also get (2).
\end{proof}

We are now fully set up. The next three claims constitute the final part of the argument.

\begin{claim} The projection $C(X\times Y)\rightarrow Y$ is surjective.
\end{claim}
\begin{proof} By Lemma \ref{plane curves always respect equivalence relation} at the point $(x,y_0)$, for any $x\in X$.
\end{proof}

\begin{claim} The projection $C(X\times Y)\rightarrow X$ is injective.
\end{claim}
\begin{proof} Fix $x\in X$. By Lemma \ref{composition with two fixed points}(2), the map $Q^x$, given by $y\mapsto Q(x,y)$, is non-constant and almost regular at all points of $Y$. By Lemma \ref{composition with two fixed points}(3), we have $Q^x(y_0)=x$. So by Corollary \ref{polynomials biject classes}, $Q^x$ bijects $Y$ to $X$. In particular, there can be only one $y\in Y$ with $Q(x,y)=x$.
\end{proof}

\begin{claim}\label{only one fixed point} $C$ defines a bijection between $X$ and $Y$.
\end{claim}
\begin{proof} By construction, $X$ and $Y$ are good, so $|X|=|Y|=k$. Now apply the previous two claims.
\end{proof}

We now finish the proof of Theorem \ref{nlm case}. By Claim \ref{only one fixed point}, the map $x\mapsto Q(x,y_0)$ has only one fixed point in $X$. By Lemma \ref{composition with two fixed points}, the same map has $k$ fixed points in $X$. We conclude that $k=1$. In particular, our $\mathscr M$-definable map $f:M\rightarrow F^n$ is generically injective. After editing finitely many points, we may assume $f$ is injective. The theorem now follows from Theorem 4.12 of \cite{AssafBen}.
\end{proof}

\section{Proof of the Main Theorem}

In the final section, we collect the results up to this point and prove Theorem \ref{thm: mainclassification}. First we show that Theorem \ref{nlm case} applies to polynomial reducts:
 
 \begin{theorem}\label{nlm classification} Let $\mathcal P$ be a set of complex polynomials, and let $\mathscr M=(\mathbb C,\mathcal P)$. If $\mathscr M$ is not locally modular, then $\mathscr M$ defines $+$ and $\times$.
 \end{theorem}
 \begin{proof} By Theorem \ref{nlm case} and Remark \ref{non-injective unary}, it suffices to recover (1) a binary polynomial that depends on both variables, and (2) a non-linear unary polynomial. We build these in the following two claims.
 \begin{claim} There is an $\mathscr M$-definable binary polynomial depending on both variables.
 \end{claim}
 \begin{proof} By non-triviality and Lemma \ref{unary implies trivial}, there is an $\mathscr M$-definable polynomial depending on at least two variables. Now specialize all but two variables generically.
 \end{proof}
 \begin{claim} There is an $\mathscr M$-definable non-linear unary polynomial.
 \end{claim}
 \begin{proof} Suppose there is no such polynomial. By non-local modularity, $\mathscr M$ is not interdefinable with a vector space, so there must be some $P\in\mathcal P$ which is not linear. By assumption, $P$ uses at least two variables. 
 
 Now by non-linearity, there is a monomial $m$ in $P$ with total degree at least 2. Clearly, then, there are two variables $(x,y)$ that each appear in $P$ and such that $m$ has total degree at least 2 in $x$ and $y$: indeed, if two variables appear in $m$ this is trivial, and if not then let $x$ be the unique variable in $m$ and let $y$ be any other variable appearing in $P$.
 
 Let $x$ and $y$ be variables as above. Then, specializing all other variables generically, we recover a binary non-linear polynomial, say $Q(x,y)$. Considering further specializations of $x$ and $y$, our assumption forces that $Q$ is linear in each of $x$ and $y$. That is, $Q(x,y)=axy+bx+cy+d$ for some constants $a,b,c,d$. By non-linearity, $a\neq 0$. But then $Q(x,x)$ is $\mathscr M$-definable and non-linear, a contradiction.
 \end{proof}
 \end{proof}

 We now deduce the main theorem:

  \begin{theorem}[Theorem \ref{thm: mainclassification}]
Suppose $\mathcal P$ is a collection of complex polynomial maps, and let $\mathscr M=(\mathbb C;\mathcal P)$. Then $\mathscr M$ is interdefinable with exactly one of following:
\begin{enumerate}[{\rm (i)}]
    \item $(\mathbb C;\mathcal U)$, where $\mathcal U$ is a collection of unary polynomials.
    \item The $F$-vector space structure on $\mathbb C$ for some subfield $F\leq\mathbb C$.
    \item $(\CC;\times_b)$ for some (unique) $b\in\mathbb C$.
    \item The full reduct $(\CC;+,\times)$.
\end{enumerate}
Moreover, (1) happens if and only if each $P\in\mathcal P$ involves at most one variable. (2) happens if and only if (1) fails and each $P\in\mathcal P$ is linear, and in this case $F$ is generated by all coefficients on all variables appearing in the $P\in\mathcal P$. (3) happens if and only if (4) fails and each $P\in\mathcal P$ is a monomial in the group operation $\times_b$. And (4) happens in every other case.
\end{theorem}
\begin{proof} If $\mathscr M$ is trivial, then by Lemma \ref{trivial implies unary}, $\mathscr M$ is interdefinable with a structure as in (1). If $\mathscr M$ is non-trivial locally modular, then by Theorem \ref{loc mod classification}, $\mathscr M$ is interdefinable with a structure as in (2) or (3). If $\mathscr M$ is not locally modular, then by Theorem \ref{nlm classification}, $\mathscr M$ is interdefinable with the structure in (4). In particular, (1) holds if and only if $\mathscr M$ is trivial; (2) or (3) holds if and only if $\mathscr M$ is non-trivial locally modular; and (4) holds if and only if $\mathscr M$ is not locally modular.

The `moreover' clause follows for (1) by Lemmas \ref{trivial implies unary} and \ref{unary implies trivial}. For (2) and (3), quantifier elimination implies that all $P\in\mathcal P$ are (respectively) linear or monomials twisted by $r$, while the converse is Lemma \ref{P: loc mod examples}. Then (4) must happen in all other cases.
\end{proof}
 
 As a result, we conclude that `almost all' complex polynomials in several variables are capable of defining each other:

\begin{theorem}\label{single poly main thm} Let $P\in\mathbb C[x_1,...,x_m]$ and $Q\in\mathbb C[x_1,...,x_n]$ be polynomials. Assume that at least two variables appear in each of $P$ and $Q$, and neither of $P$ and $Q$ is linear or a twisted monomial. Then there is a first-order definition of $Q$ in terms of $P$, and vice versa.
\end{theorem}

\begin{proof} By Theorem \ref{thm: mainclassification}, the structure $(\mathbb C,P)$ defines addition and multiplication, and therefore defines $Q$. The other direction holds symmetrically.
\end{proof}

We ends with a remark on the possibility of further generalization.

\begin{remark}
Analogues of Theorem~\ref{nlm classification} are expected in many other settings (e.g., with polynomials replaced by rational function).
Except for the material of  Section~6, our analysis goes through in these settings. The only missing ingredient is a suitable counterpart of  Fact~\ref{C: application of Minh} in such setting. It is expected that this can be obtained from the recent result in~\cite{chernikov2023modeltheoretic} together with an analysis as in~\cite{JTR}.
\end{remark}
\bibliographystyle{amsalpha}
\bibliography{reference}

\end{document}